\documentclass[a4paper, oneside, 11pt]{amsart}
\usepackage{amssymb}
\usepackage{amsmath}
\usepackage{amsfonts}
\usepackage{amsthm}
\usepackage{subfigure}
\usepackage{epsfig}
\usepackage[all]{xy}
\usepackage[utf8]{inputenc}
\usepackage{mathabx}
\usepackage{graphicx}
\usepackage{xcolor}
\usepackage[T1]{fontenc}

\usepackage{imakeidx}
\usepackage[refpage]{nomencl}
\makeindex[columns=3]
\makenomenclature
\usepackage[pageanchor]{hyperref}

\DeclareFontFamily{T1}{calligra}{}
\DeclareFontShape{T1}{calligra}{m}{n}{<->s*[1.5]callig15}{}
\DeclareMathAlphabet\mathcalligra   {T1}{calligra} {m} {n}
\DeclareMathAlphabet\mathzapf       {T1}{pzc} {mb} {it}
\DeclareMathAlphabet\mathchorus     {T1}{qzc} {m} {n}
\DeclareMathAlphabet\mathrsfso      {U}{rsfso}{m}{n}

\usepackage{comment}
\usepackage{stmaryrd}
\usepackage{enumitem}
\usepackage{dynkin-diagrams}
\usepackage[scr=boondox]{mathalpha}
\usepackage{hyperref} 
\hypersetup{colorlinks=true,linkcolor=blue,citecolor=red,urlcolor=blue}

\newtheorem{theorem}{Theorem}[section] 
\newtheorem{lemma}[theorem]{Lemma}     
\newtheorem{corollary}[theorem]{Corollary}
\newtheorem{proposition}[theorem]{Proposition}

\theoremstyle{definition}
\newtheorem{example}[theorem]{Example}
\newtheorem{remark}[theorem]{Remark}

\title[Maximal Unipotent subgroups of $\mathcal{S}$-arithmetic subgroups]{Arithmetic subgroups of Chevalley group schemes over function fields II: Conjugacy classes of maximal unipotent subgroups}

\newcommand{\stab}{\operatorname{Stab}}
\newcommand{\fix}{\operatorname{Fix}}

\newcommand{\SecD}{\operatorname{Ch}\left( \partial_\infty^{k} X_\mathcal{S} \right)}
\newcommand{\p}{P}

\newcommand{\card}{\operatorname{Card}}
\newcommand{\OS}{\mathcal{O}_\mathcal{S}}
\newcommand{\X}{X_\mathcal{S}}
\newcommand{\D}{\partial_{\infty} D}
\begin{document}

\maketitle

\author{\textbf{Claudio Bravo}
\footnote{Centre de Mathématiques Laurent Schwartz, École Polytechnique, Institut Polytechnique de Paris, 91128 Palaiseau Cedex, France. Email: \email{claudio.bravo-castillo@polytechnique.edu}}
}
\author{\textbf{Benoit Loisel}
\footnote{Université de Poitiers (Laboratoire de Mathématiques et Applications, UMR7348), Poitiers, France. Email: \email{benoit.loisel@math.univ-poitiers.fr}}
}

\vspace{0.5cm}

\begin{abstract}
Let $\mathcal{C}$ be a smooth, projective, geometrically integral curve defined over a perfect field $\mathbb{F}$. Let $k=\mathbb{F}(\mathcal{C})$ be the function field of $\mathcal{C}$.
Let $\mathbf{G}$ be a split simply connected semisimple $\mathbb{Z}$-group scheme.
Let $\mathcal{S}$ be a finite set of places of $\mathcal{C}$.
In this paper, we investigate on the conjugacy classes of maximal unipotent subgroups of $\mathcal{S}$-arithmetic subgroups.
These are parameterized thanks to the Picard group of $\OS$ and the rank of $\mathbf{G}$.
Furthermore, these maximal unipotent subgroups can be realized as the unipotent part of natural stabilizer, which are the stabilizers of sectors of the associated Bruhat-Tits building.
We decompose these natural stabilizers in terms of their diagonalisable part and unipotent part, and we precise the group structure of the diagonalisable part.
\\


\end{abstract}


\section{Introduction}\label{section intro}

Let $\mathcal{C}$ be a smooth, projective, geometrically integral curve defined over a perfect field $\mathbb{F}$.
Let $\mathcal{S}$ be a finite set of closed points of $\mathcal{C}$.
We denote by $\mathcal{O}_{\mathcal{S}}$ the ring of rational functions of $\mathcal{C}$ that are regular outside $\mathcal{S}$, so that $\mathrm{Quot}(\OS)=k$.

In this article, all the considered group schemes are assumed to be linear, smooth and connected.
For such a group scheme $\mathbf{H}$ defined over a (commutative) ring $R$, we denote by $\mathbf{H}(R)$ the group of $R$-points of $\mathbf{H}$.

Let $\mathbf{G}_k$ be a split simply connected semisimple algebraic group defined over $k$.
Since $\mathbf{G}_k$ is assumed split simply connected and semisimple, it follows from \cite[Exp.~XXV, Cor.~1.3]{SGA3-3} that there exists a semisimple $\mathbb{Z}$-group scheme $\mathbf{G}$ such that $\mathbf{G} \otimes_\mathbb{Z} k \cong \mathbf{G}_k$.
Moreover, since $\mathbf{G}$ is defined over $\mathbb{Z}$, the group of $\OS$-points $\mathbf{G}(\OS)$ of $\mathbf{G}$ makes sense.

Recall that two subgroups $H_1, H_2$ of a group $H$ are called commensurable exactly when they contain a common finite index subgroup $H_{1,2}$, i.e., $H_{1,2} \subseteq H_1 \cap H_2$ and $[H_{1,2}: H_1]$, $[H_{1,2}: H_1] <\infty$.
A subgroup $G \subset \mathbf{G}(k)$ that is commensurable to $\mathbf{G}(\OS)$ is called an $\mathcal{S}$-arithmetic subgroup of $\mathbf{G}(k)$.
In~\cite[\S 1]{Raghunathan}, $\mathcal{S}$-arithmetic subgroups are defined as follows:
considering a faithful linear representation $\rho: \mathbf{G}_k \to \mathrm{SL}_{n,k}$ defined over $k$,
an $\mathcal{S}$-arithmetic subgroup is a subgroup of $\mathbf{G}(k)$ which is commensurable with $\mathbf{G}(k) \cap \rho^{-1} \left( \mathrm{SL}_n(\OS)\right)$.
According to~\cite[\S 1.4.5]{BT2}, $\mathbf{G}$ admits a faithful linear representation defined over $\mathbb{Z}$ which is a closed embedding. Thus, the natural inclusion $\mathbf{G}(\OS) \subset \mathbf{G}(k) \cap \rho^{-1} \left( \mathrm{SL}_n(\OS)\right)$ is an equality.
Hence, the definition of $\mathcal{S}$-arithmetic subgroups given by~\cite[\S 1]{Raghunathan} coincides with that of commensurable subgroups of $\mathbf{G}(\OS)$ given above.
Both definitions depend on either the choice of a representation or of a model of $\mathbf{G}$ over $\mathbb{Z}$.

In the following, we will consider certain subgroups of $\mathbf{G}(k)$ consisting in semisimple elements, or in unipotent elements. This makes sense from the Jordan decomposition of the linear algebraic group $\mathbf{G}_k$.
We say that a subgroup $U$ of $\mathbf{G}(k)$ is unipotent (resp. diagonalisable) if any element $u\in U$ is unipotent (resp. semisimple) in $\mathbf{G}(k)$.
In this paper, we focus on maximal unipotent subgroups of $\mathcal{S}$-arithmetic subgroups.

Let $F$ be a field of characteristic $0$ or a field of positive characteristic $p>0$ such that $[F:F^p] \leq p$.
For instance, $F$ can be a global field.
Let $U$ be a unipotent subgroup of $\mathbf{G}(F)$.
Then, there exists a (maximal) unipotent subgroup-scheme $\mathbf{U}_F$ of $\mathbf{G}_F$ such that $U \subseteq \mathbf{U}_F(F)$. This result is a consequence of~\cite[Cor.~3.7]{BoTi2}, when $F$ has characteristic zero, and of Theorem~\cite[Thm.~2]{G}, in positive characteristic.\footnote{This also holds without the assumption that $\mathbf{G}$ is split.}
In particular, if $U$ is a maximal unipotent subgroup of $\mathbf{G}(F)$, then $U=\mathbf{U}_F(F)$.
Since $\mathbf{U}_F(F)$ is solvable, it is contained in the group $\mathbf{B}_F(F)$, for a Borel subgroup $\mathbf{B}_F$ of $\mathbf{G}_F$.\footnote{This also holds when $\mathbf{G}$ is assumed quasi-split instead of split.}
Moreover, the group $\mathbf{U}_F(F)$ is the unipotent radical of $\mathbf{B}_F(F)$.
Since the Borel subgroups of $\mathbf{G}_F$ are $\mathbf{G}(F)$-conjugates, we conclude that all the maximal unipotent (abstract) subgroups of $\mathbf{G}(F)$ are $\mathbf{G}(F)$-conjugates.

When $R$ is an arbitrary commutative ring, maximal unipotent subgroups are not always $\mathbf{G}(R)$-conjugate.
Moreover, in the context of $\mathcal{S}$-arithmetic subgroups, Corollary~\ref{coro number of max unip sub for GA} will provide a family of counter-examples, depending on the Picard group $\mathrm{Pic}(\OS)$.
Nevertheless, when the arithmetic properties of the ring $\OS$ are close to that of a field, e.g., when $\OS$ is a principal ideal domain, one will observe that the maximal unipotent subgroups are $\mathbf{G}(\OS)$-conjugates (cf.~Corollary~\ref{coro max unip sub A principal}).

Assuming that $\mathbb{F}$ is a finite field of characteristic $p$, the problem of classify the maximal unipotent subgroups of an $\mathcal{S}$-arithmetic group $G$ is analogous to the problem of classifying $p$-Sylow subgroups of finite groups.
Indeed, in this context any $p^r$-torsion element $g \in G$ is unipotent since $g^{p^k}=\mathrm{id}$ exactly when $(g-\mathrm{id})^{p^k}=0$. 
Thus, if $g \in G$ has a $p$-power order, then Lemma~\ref{lemma unipotents in G} together with Theorem~\ref{main teo 1} implies that $g$ is contained in a maximal unipotent subgroup.
Conversely, any unipotent element of $G$ has a $p$-power order.
Therefore, we conclude that the maximal unipotent subgroup of $G$ are exactly the maximal subgroups containing torsion elements of $p$-power order: these are the ``$\mathcal{S}$-arithmetical'' analogous of $p$-Sylow subgroups of finite groups.

This work is devoted to understanding the conjugacy classes of maximal unipotent subgroups contained in any $\mathcal{S}$-arithmetic subgroup $G$ of $\mathbf{G}(k)$.
As we mentioned above, for simply connected semisimple groups over $F$, there is a unique conjugacy class of maximal unipotent subgroups. This is a consequence of the $\mathbf{G}(F)$-conjugacy of Borel subgroups.
This $\mathbf{G}(F)$-conjugacy result corresponds to the $\mathbf{G}(F)$-transitive action on chambers of the spherical Tits building associated to $(\mathbf{G},F)$ \cite[Thm.~6.56]{AB}. 
Using this combinatorical interpretation,
in \S~\ref{section maximal unipotents} we describe the maximal unipotent subgroups of $G$, as well as their $G$-conjugacy classes, in terms of certain unipotent subgroups $U(G_D)$ of the $G$-stabilizers $G_D$ of chambers $\D$ of the spherical building defined from $(\mathbf{G},k)$.
These properties are summarized in Theorem~\ref{main teo 1}.
Theorem~\ref{main teo 3} provides a parametrization of the $G$-conjugacy classes of maximal unipotent subgroups of $G$ in terms of the \'Etale cohomology, as developed in \S~\ref{section prove of 2.2}.

In some cases (cf.~\cite{MS1},~\cite{MS2},~\cite[Ch.~II, \S~2.5]{S},~\cite[\S 4 (d)]{Be} and~\cite[\S 3]{Stuhler}), the groups $G_D$ (or $G$-stabilizers of simplices in the affine building which are related to $G_D$) have been directly computed.
So, it appears to be more natural to compute directly the $G_D$ than the $U(G_D)$.
Then, in order to understand the groups $U(G_D)$, in \S~\ref{section fixing} and \S~\ref{section proof of 2.3}, we compare each $G_D$ with its (unique) maximal unipotent subgroup $U(G_D)$.
Indeed, in Theorem~\ref{main teo 2}, we exhibit a diagonalizable subgroup of $\mathbf{G}(k)$ that is isomorphic to $G_D/U(G_D)$.
This diagonalizable group decomposes as the direct product of an ``arithmetic bounded torus'' $T$ with a finitely generated free $\mathbb{Z}$-module.
When $\mathbb{F}$ is finite, the group $T$ is finite.
Therefore, the preceding result describes the torsion and the free part of the abelian group $G_D/U(G_D)$, which is finitely generated in this case.

Assuming that $\mathbb{F}$ is finite, in \S~\ref{section examples}, we apply Theorems~\ref{main teo 1}, \ref{main teo 3} and \ref{main teo 2} to principal congruence subgroup $\Gamma$ of $\mathbf{G}(\OS)$, which are arithmetic subgroups of $\mathbf{G}(k)$ whose torsion is $p$-primary.
More precisely, we show that, when $\mathcal{S} = \{ \p \}$, the maximal unipotent subgroups of $\Gamma$ are exactly the stabilizers in $\Gamma$ of chambers of the spherical building defined from $(\mathbf{G},k)$.
We also characterize the preceding unipotent subgroups when $\OS$ is a principal ideal domain or when $\mathbf{G}=\mathrm{SL}_n$ (See Corollary~\ref{corollary principal cong sub and principal ideal domains} and Example~\ref{ex max unip sln}).


Consider a (non necessarily split) semisimple group $\mathbf{G}$ of rank $1$ defined over a local field $K$.
Since the group of rational points $\mathbf{G}(K)$ of $\mathbf{G}$ is a locally compact unimodular group, it has a Haar measure $\mu$ which is $\mathbf{G}(K)$-invariant.
A lattice $\Lambda$ of $\mathbf{G}(K)$ is a subgroup of finite $\mu$-covolume, i.e., $\mu(\mathbf{G}(K)/\Lambda)$ is finite. 
In~\cite[Thm.~2.3]{Ba03}, Baumgartner characterizes the maximal unipotent subgroups of any lattice $\Lambda$ of $\mathbf{G}(K)$.

Assume that $\mathbb{F}$ is finite and that $\mathbf{G}$ has rank $1$ (i.e., $\mathbf{G}=\mathrm{SL}_2$ since we assume that $\mathbf{G}$ is split simply connected and semisimple).
It follows from~\cite[Ch. II, \S 2.9, Ex.~2, Pag. 110]{S} that $\mathbf{G}(\mathcal{O}_{\lbrace \p \rbrace})$ is a lattice of $\mathbf{G}(k_{ \p })$, where $k_P$ is the completion of $k$ at $P$.
Since $\mathbf{G}(\mathcal{O}_{\lbrace \p \rbrace}) \subseteq \mathbf{G}(\OS)$, when $\p \in \mathcal{S}$, the group $\mathbf{G}(\OS)$ is also a lattice of $\mathbf{G}(k_{ \p })$.
In particular, any $\mathcal{S}$-arithmetic subgroup $G$ of $\mathbf{G}(k_{ \p })$ is a lattice of $\mathbf{G}(k_{ \p })$.
Thus, the results of Baumgartner describe the maximal unipotent subgroups contained in $G$ when $\mathbb{F}$ is finite and $\mathbf{G}$ has rank $1$.
In particular, Theorems~\ref{main teo 1}, \ref{main teo 3} and~\ref{main teo 2} extend the aforementioned results due to Baumgartner to the context of (non necessity finite) perfect fields and algebraic $k$-groups of higher rank.

Still in the context of finite fields, Serre proves in~\cite[Ch. II, \S~2.9]{S} that the maximal unipotent subgroups contained in finite index subgroups $G$ of $\mathrm{SL}_2(\mathcal{O}_{\lbrace \p\rbrace})$, whose torsion is $p$-primary, are exactly the stabilizers of chambers of the spherical building of $(\mathrm{SL}_2,k)$.
In other words, Serre proves that the aforementioned unipotent subgroups are the $G$-stabilizers defined from the action of $G$ on $\mathbb{P}^1(k)$ by Moebius transformations.
In the same work, Serre characterizes the homology of $G$ modulo a representative system of the conjugacy classes of its maximal unipotent subgroups in terms of the Euler-Poincar\'e characteristic of $G$.
These results are summarized in~\cite[Thm. 14, Ch. II, \S 2.9]{S}. 
In~\cite[\S 3, Pag. 155]{Se2}, Serre conjectures that these results can be extended to the context where $\mathbf{G}$ has arbitrary rank.

In the following, we describe the groups involved in the preceding conjecture due to Serre.
For simplicity, this study is limited to split groups because they admit a Chevalley pinning defined over $\mathbb{Z}$.
In particular, points of tori and root groups naturally make sense over arbitrary commutative rings.
An investigation for the quasi-split non-split group $\mathrm{SU}_3$ is developed in~\cite{CB}.
Since $\mathrm{SU}_3$ and split groups encode the behaviour of all the possible Borel subgroups of quasi-split reductive groups, it is expectable that one can obtain results in the general case from these two cases.


\section{The Bruhat-Tits building and the action of \texorpdfstring{$\mathbf{G}(k)$}{G(k)}}\label{section preliminaries}

\subsection{Recollections on algebraic groups and Bruhat-Tits buildings}\label{section recall alg groups and BT}

In the following, we consider a given Killing couple $(\mathbf{T},\mathbf{B})$ of $\mathbf{G}$ defined over $\mathbb{Z}$, which is a couple consisting in a Borel subgroup $\mathbf{B}$ and a maximal torus $\mathbf{T}$ contained in $\mathbf{B}$.
Such a Killing couple exists according to~\cite[Exp.~XXII, Def.~5.3.13]{SGA3-3} since $\mathbf{G}$ is assumed to be split.
The Borel subgroup $\mathbf{B}$ defines a subset of positive roots $\Phi^+=\Phi(\mathbf{T},\mathbf{B})$ of the set of root $\Phi(\mathbf{T},\mathbf{G})$ as in~\cite[Thm.~14.8 and \S 21.8]{BoA}.
This induces a basis of simple roots $\Delta = \Delta(\mathbf{B})$~\cite[VI.1.6]{Bourbaki} of $\mathbf{G}$ relatively to the Borel $k$-subgroup $\mathbf{B}$.
For any $\alpha \in \Phi$, let $\mathbf{U}_{\alpha}$ be the $\mathbf{T}$-stable unipotent subgroup of $\mathbf{G}$ defined from this.
Since $\mathbf{G}$ is assumed to be split, it admits a Chevalley pinning as defined in \cite[Exp.~XXIII, Def.~1.1]{SGA3-3}.
In the sequel, we denote by $\theta_{\alpha}: \mathbb{G}_{a} \rightarrow \mathbf{U}_{\alpha}$ the $\mathbb{Z}$-isomorphism given by the Chevalley pinning.
We denote by $\mathbf{U}^{+}$ the subgroup of $\mathbf{G}$ generated by the $\mathbf{U}_{\alpha}$ for $\alpha \in \Phi^{+}$.
The group $\mathbf{U}^+$ is the unipotent radical of $\mathbf{B}$ according to~\cite[Exp.~XXVI, Prop.~1.12(i)]{SGA3-3}.

For each $\p \in \mathcal{S}$,
let $\nu_{\p}:k \to \mathbb{Z} \cup \lbrace \infty \rbrace$ be the valuation induced by $\p$, and let $k_{\p}$ be the completion of $k$ with respect to $\nu_{\p}$.
We denote by $\mathcal{O}_{\p}$ the ring of integers of the complete valued field $k_{\p}$,
that is $\mathcal{O}_{\p}=\lbrace x \in k_{\p}: \nu_{\p}(x) \geq 0 \rbrace$.

The datum of $\mathbf{T}$ and the root groups $\mathbf{U}_\alpha$, together with the valuation $\nu_\p$, induce a root group datum on $\mathbf{G}(k_\p)$ according to~\cite[\S 6.2.3(b)]{BT1}.
It follows from~\cite[\S 7.4.2]{BT1} that this datum defines an affine building $X_{\p}=X(\mathbf{G}, k, {\p})$, called the Bruhat-Tits building of the $k$-group $\mathbf{G}_k$ over the valued field $(k, \nu_{\p})$, together with an apartment $\mathbb{A}_{0,\p}$ of $X_\p$.
This apartment $\mathbb{A}_{0,\p}$ is a Euclidean space over a vector space $V_{0,\p}$.
We denote by $D_{0,\p} \subset V_{0,\p}$ the vector chamber in $V_{0,\p}$
associated to $(\mathbf{T},\mathbf{B})$, which can be defined by the formula $D_{0,\p}= \{ x \in V_{0,\p} : \alpha(x) > 0, \forall \alpha \in \Delta \}$ as in~\cite[\S 1.3.12]{BT1}.
Since $\mathbf{G}_k$ is semisimple and simply connected, it follows from \cite[\S 6.4.16(b)]{BT1} that the pointwise stabilizer of $\mathbb{A}_{0,\p}$ is $\mathbf{T}(\mathcal{O}_\p)$.
Moreover, since $\mathbf{G}_{k}$ is split, semisimple and simply connected, it follows from~\cite[\S 4.6.31 \& \S 4.6.32]{BT2} that there exists a special vertex $x_{0,\p} \in X_\p$ such that its stabilizer in $\mathbf{G}(k_\p)$ is $\mathbf{G}(\mathcal{O}_\p)$.
In fact, since $\mathbf{T}(\mathcal{O}_\p) \subset \mathbf{G}(\mathcal{O}_\p)$, the vertex $x_{0,\p}$ belongs to $\mathbb{A}_{0,\p}$~\cite[\S 7.4.10]{BT1}.

Any apartment $\mathbb{A}$ of $X_\p$ is endowed with an affine Coxeter complex structure associated to $\Phi$.
It induces a spherical Coxeter complex structure on the underlying vector space $V$ of $\mathbb{A}$. 
The vector chambers of those $V$ are called the vector chambers of $X_\p$.

\subsection{The diagonal action of $\mathbf{G}(k)$}\label{section diagonal action}

We denote by $\X$ the direct product of the buildings $X_\p$, for $\p \in \mathcal{S}$.
It is a Euclidean building of type $\Phi^{\mathcal{S}}$, as a finite product of Euclidean buildings all of type $\Phi$.
The abstract group $\widehat{G}_{\mathcal{S}}:=\prod_{\p \in \mathcal{S}} \mathbf{G}(k_{\p})$ acts on $X_\mathcal{S}$ via:
\[(g_{\p})_{\p \in \mathcal{S}} \cdot (x_{\p})_{\p \in \mathcal{S}} =(g_{\p} \cdot x_{\p})_{\p \in \mathcal{S}}, \quad \forall (g_{\p})_{\p \in \mathcal{S}} \in \hat{G}_{\mathcal{S}}, \, \, \forall (x_{\p})_{\p \in \mathcal{S}} \in \X.\]
By definition, an apartment (resp. a vector chamber) of $\X$ is a product of apartments (resp. vector chambers) of the buildings $X_\p$, for $\p \in \mathcal{S}$.
Recall that $\mathbf{G}(k_\p)$ acts strongly transitively on the spherical building structure associated to $X_\p$ according to \cite[\S 2.2.6]{BT1},
i.e., $\mathbf{G}(k_\p)$ acts transitively on the set of pairs $(\mathbb{A},D)$ consisting of an apartment $\mathbb{A} \subset X_{\p}$ and a vector chamber $D \subset \mathbb{A}$.
This transitivity property extends to the product action of $\widehat{G}_\mathcal{S}$ on $\X$ seen as a product of buildings.

We define an apartment of $\X$ by setting $\mathbb{A}_0 := \prod_{\p \in \mathcal{S}} \mathbb{A}_{0,\p}$.
Since $\widehat{G}_\mathcal{S}$ acts transitively on the set of apartments of $\X$,
the apartments of $\X$ are the
$g \cdot \mathbb{A}_0$, for $g \in \widehat{G}_\mathcal{S}$.
Hence, the group $\widehat{G}_\mathcal{S}$ acts transitively on the vector chambers of $\X$, that are the vector chambers of the vector spaces associated to its apartments.

The conical cell $D_0 = \prod_{\p \in \mathcal{S}} D_{0,\p}$ is a vector chamber of the underlying vector space $V_0$ of $\mathbb{A}_0$.
By transitivity of $\widehat{G}_\mathcal{S}$,
the vector chambers of $\X$ are
the $D = g \cdot D_0 \subset g \cdot V_0$, for $g \in \widehat{G}_\mathcal{S}$.
For any $x \in \mathbb{A}_0$, the subset $Q(x,D_0) = x+D_0 \subset \mathbb{A}_0$ is a sector chamber as a product of sector chambers on each component.
As in~\cite[\S 7.4.12]{BT1}, we define the sector chambers of $X$ as the subsets $g \cdot Q(x,D_0) \subset g \cdot \mathbb{A}_0$, for $g \in \widehat{G}_\mathcal{S}$ and $x \in \mathbb{A}_0$.

We fix a point $x_0 = (x_{0,\p})_{\p \in \mathcal{S}}$ of the apartment $\mathbb{A}_0$.
It is a special vertex since so are the $x_{0,\p} \in \mathbb{A}_{0,\p}$, for all $\p \in \mathcal{S}$.

Since $\mathbf{G}(k)$ embeds in $\hat{G}_{\mathcal{S}}$ via $g \mapsto (g)_{\p \in \mathcal{S}}$,
the group $\mathbf{G}(k)$ acts diagonally on $\X$ via:
\[g \cdot (x_{\p})_{\p \in \mathcal{S}} =(g \cdot x_{\p})_{\p \in \mathcal{S}},\]
for $(x_{\p})_{\p \in \mathcal{S}} \in \X$ and $g \in \mathbf{G}(k)$.

\subsection{Rational chambers at infinity}\label{section rational chambers}

We denote by $\partial_\infty \X$ the spherical building at infinity of $\X$ as defined in~\cite[\S 11.8]{AB}.
It consists of parallelism classes of geodesical rays of $\X$.
Its apartments (resp. chambers) are in one-to-one correspondence with the apartments (resp. vector chambers) of $\X$, as proved in \cite[Lemma 11.75 \& Thm. 11.79]{AB}.
Since $\widehat{G}_\mathcal{S}$ acts simplicially and by isometries on $\X$, it sends geodesical rays onto geodesical rays.
Thus, it induces an action of $\widehat{G}_\mathcal{S}$ on $\partial_\infty \X$.
We denote by $\partial_\infty \mathbb{A}$ (resp. $\partial_\infty D$) the image in $\partial_\infty \X$ of an apartment $\mathbb{A}$ (resp. a vector chamber $D$).
Since $\widehat{G}_{\mathcal{S}}$ acts strongly transitively on $\X$, it also acts strongly transitively on $\partial_\infty \X$.

\begin{lemma}\label{lemma stab of D_0}
The stabilizer in $\mathbf{G}(k)$ of $\D_0$ is $\mathbf{B}(k)$.
\end{lemma}

\begin{proof}
For each $\p \in \mathcal{S}$, the stabilizer of $\D_{0,\p}$ in $\mathbf{G}(k_\p)$ is $\mathbf{B}(k_\p)$.
Hence the stabilizer of $\D_0$ in $\widehat{G}_\mathcal{S}$ is $\prod_{\p \in \mathcal{S}} \mathbf{B}(k_\p)$.
Thus, by diagonal action, the stabilizer of $\D_0$ in $\mathbf{G}(k)$ is the intersection of the groups $\mathbf{B}(k_\p) \cap \mathbf{G}(k) \subset \mathbf{G}(k_\p)$, for $\p \in \mathcal{S}$.
Since the algebraic group $\mathbf{B}$ is closed in $\mathbf{G}$, we have that $\mathbf{B}(k_\p) \cap \mathbf{G}(k) = \mathbf{B}(k)$, for each $\p \in \mathcal{S}$.
\end{proof}

We denote by $\SecD$ the $\mathbf{G}(k)$-orbit of $\D_0$ in the set of chambers of $\partial_{\infty} \X$.
An element of $\SecD$ is called a $k$-rational chamber at infinity of $\X$.
Since $\mathrm{Stab}_{\mathbf{G}(k)}(\D_0)$ equals $\mathbf{B}(k)$, there exists a $\mathbf{G}(k)$-equivariant one-to-one correspondence between $\SecD$ and the Borel variety $\mathbf{G}(k)/\mathbf{B}(k)$.
Thus, the set $\SecD$ corresponds to the set of chambers in the usual spherical building defined from $(\mathbf{G},k)$.

\section{Main results}\label{section max unip sub}

Recall that $\mathbf{G}_k$ denotes a split simply connected semisimple $k$-group, and $\mathbf{G}$ denotes a $\mathbb{Z}$-model of this algebraic $k$-group, i.e., a semisimple $\mathbb{Z}$-group such that $\mathbf{G} \otimes_{\mathbb{Z}} k \cong \mathbf{G}_k$.
The group $\mathbf{B}$ is a Borel subgroup of $\mathbf{G}$, whose unipotent radical is $\mathbf{U}^+$.
Recall also that the base field $\mathbb{F}$ is assumed to be perfect.
The main results of this work is the following theorem that describes maximal unipotent subgroups of an $\mathcal{S}$-arithmetic subgroup $G \subset \mathbf{G}(k)$ in terms of its action on $\SecD$.

\begin{theorem}\label{main teo 1}
Let $G$ be an $\mathcal{S}$-arithmetic subgroup of $\mathbf{G}(k)$.
For each $\D \in \SecD$, let $h_D \in \mathbf{G}(k)$ be an arbitrary element such that $h_D \cdot \D= \D_0$. Let us write:
\begin{equation}\label{eq G_s and borel}
G_{D}= \mathrm{Stab}_{G}(\D) = h^{-1}_{D} \mathbf{B}(k) h_{D} \cap G, \quad \text{ and } \quad U(G_{D}) :=h^{-1}_{D} \mathbf{U}^{+}(k) h_{D} \cap G.
\end{equation}
Then:
\begin{itemize}
\item[(1)] $U(G_{D})$ is the subgroup consisting in all the unipotent elements in $G_{D}$, and
\item[(2)] $\mathfrak{U}:=\lbrace U(G_{D}) : \D \in \SecD \rbrace$ is the set all the maximal unipotent subgroups of $G$.
\end{itemize}

Fix a set $\lbrace \D_{\sigma}: \sigma \in \Sigma \rbrace$ of representatives of the $G$-orbits of $\SecD$, and write $h_{\sigma}:=h_{D_{\sigma}}$, $G_{\sigma}=G_{D_{\sigma}}$ and $U(G_{\sigma})=U(G_{D_{\sigma}})$.
Then:
\begin{enumerate}
\item[(3)] $\mathfrak{U}/G:=\lbrace U(G_{\sigma}) : \sigma \in \Sigma \rbrace$ is a set of representatives of the conjugacy classes of maximal unipotent subgroups of $G$.
\end{enumerate}
\end{theorem}

The isomorphism classes of line bundles on $\mathrm{Spec}(\OS)$ (resp. $\mathcal{C}$) form a group, with the tensor product as composition law, which is called the Picard group and denoted by $\mathrm{Pic}(\OS)$ (resp. $\mathrm{Pic}(\mathcal{C})$).
Since $\mathcal{C}$ has dimension one, the group
$\mathrm{Pic}(\OS)$ coincides with the ideal class group $\mathrm{Cl}(\OS)$ of the Dedekind domain $\OS$.
Moreover, if $\overline{\mathcal{S}}$ is the image of $\mathcal{S}$ in $\mathrm{Pic}(\mathcal{C})$, then, we have $\mathrm{Pic}(\OS) \cong \mathrm{Pic}(\mathcal{C}) / \langle \overline{ \mathcal{S} } \rangle$.
Using \'Etale cohomology, a numbering of the maximal unipotent subgroups described in Theorem~\ref{main teo 1} can be done for $\mathbf{G}(\mathcal{O}_{\mathcal{S}})$ in terms of $\mathrm{Pic}(\OS)$ and of the rank $\mathbf{t}=\mathrm{rk}(\mathbf{G})$ of $\mathbf{G}$, which is the dimension of $\mathbf{T}$. 
This also has interesting consequences on the $\mathcal{S}$-arithmetic subgroups $G$ of $\mathbf{G}(k)$. For instance, this implies that the number of $G$-conjugacy classes of maximal unipotent subgroup is finite whenever $\mathbb{F}$ is finite.

\begin{theorem}\label{main teo 3}
There exists a bijection between the set of $\mathbf{G}(\mathcal{O}_\mathcal{S})$-orbits in $\SecD$ and $\mathrm{Pic}(\mathcal{O}_{\mathcal{S}})^{\mathbf{t}}$.
\end{theorem}

\begin{remark}
If $\mathbf{G}_k$ is a split reductive $k$-group (not necessarily simply connected semisimple),
then there exists a bijection between the set of $\mathbf{G}(\mathcal{O}_\mathcal{S})$-orbits in $\SecD$ and $\ker \left( H^1_{\text{\'et}}(\mathrm{Spec}(\mathcal{O}_{\mathcal{S}}), \mathbf{T}) \to H^1_\text{\'et}(\mathrm{Spec}(\mathcal{O}_{\mathcal{S}}), \mathbf{G}) \right)$. This is proven in \S~\ref{section prove of 2.2}.
\end{remark}

The consequences of Theorem~\ref{main teo 2} on the number of conjugacy classes of maximal unipotent subgroups, which we study in \S~\ref{section consequences}, go further than the previously described results.
Indeed, as we explain in Corollary~\ref{coro max unip sub A principal}, this implies that each unipotent subgroup of $\mathbf{G}(\OS)$ is contained in a $\mathbf{G}(\OS)$-conjugate of the group $\mathbf{U}^{+}(\OS)$ exactly when $\OS$ is a principal ideal domain.
This shows that the description of maximal unipotent subgroups of the arithmetic group $\mathbf{G}(\OS)$ is analogous to the description of those of $\mathbf{G}(k)$, when $\OS$ is a principal ideal domain.

Let $p$ be a prime number.
We say that the torsion of a group $G \subset \mathbf{G}(k)$ is $p$-primary when each finite order element in $G$ has $p$-power order.
For instance, if $\mathbb{F}$ is finite of characteristic $p$, then, Lemma~\ref{lemma torsion elements in Gamma} shows that the torsion of any principal congruence subgroup of $\mathbf{G}(\mathcal{O}_{\mathcal{S}})$ is $p$-primary.
The following result describes the group $G_D/U(G_D)$ as the direct product of an ``arithmetic bounded torus'' $T$ and a finitely generated free $\mathbb{Z}$-module.

\begin{theorem}\label{main teo 2}
Let $G$ be an $\mathcal{S}$-arithmetic subgroup of $\mathbf{G}(k)$.
For each ${\D \in \SecD}$, let $h_D$, $G_D$ and $U(G_{D})$ as in Theorem~\ref{main teo 1}.
Then, for each ${\D \in \SecD}$, there exists $T(G_{D}) \subset h_{D}^{-1}\mathbf{T}(k) h_{D}$ such that:
\begin{itemize}
\item[(1)] $G_{D}$ is an extension of the maximal unipotent group $U(G_{D})$ by $T(G_{D})$.
\item[(2)] $T(G_{D}) \cong T \times \mathbb{Z}^r$, where $T$ is commensurable with a subgroup of $\mathbf{T}(\mathbb{F})$, and $r=r(\mathbf{G}, \mathcal{S},G,D)$ is less than or equal to $ \mathbf{t} \cdot \sharp \mathcal{S}$.
\item[(3)] Moreover, if $G \subseteq \mathbf{G}(\OS)$, then $T \subseteq \mathbf{T}(\mathbb{F})$,
\item[(4)] if $\mathcal{S}=\lbrace \p \rbrace$ and $G \subseteq \mathbf{G}(\OS)$, then $r=0$, and
\item[(5)] if $\mathcal{S}=\lbrace \p \rbrace$, $G \subseteq \mathbf{G}(\OS)$, $\mathbb{F}$ is finite of characteristic $p$ and the torsion of $G$ is $p$-primary, then $T(G_D)=\lbrace \mathrm{id} \rbrace$. In other words, the group $G_D$ is unipotent.
\end{itemize}
\end{theorem}

Note that Theorem~\ref{main teo 3}(3) implies that the group $T$ is finite whenever $\mathbb{F}$ is finite.
Thus, Theorem~\ref{main teo 3}(2) describes the decomposition in the free part and the torsion part of the finitely generated abelian group $G_D/U(G_D)$ whenever $\mathbb{F}$ is finite.

Assuming that $\mathbb{F}$ is finite, in \S~\ref{section examples}, we apply Theorems~\ref{main teo 1}, \ref{main teo 3} and \ref{main teo 2} in order to characterize the maximal unipotent subgroups of the principal congruence subgroup $\Gamma$ of $\mathbf{G}(\OS)$.
For instance, we prove that, when $\mathcal{S} = \{ \p \}$, these unipotent subgroups are exactly the $\Gamma$-stabilizers of the chambers of the spherical building defined from $(\mathbf{G},k)$, i.e., the groups $G_D$ defined in Theorem~\ref{main teo 1}.
In Corollary~\ref{corollary principal cong sub and principal ideal domains} and Example~\ref{ex max unip sln}, we precisely describe the preceding unipotent subgroups when $\OS$ is a principal ideal domain or when $\mathbf{G}=\mathrm{SL}_n$.

\section{Enumerating the conjugacy classes of maximal unipotent subgroups}\label{section consequences}

As above, we assume that the ground field $\mathbb{F}$ is perfect and that $\mathbf{G}_k$ is a split simply connected semisimple $k$-group.
By using Theorems~\ref{main teo 1} and~\ref{main teo 3}, we can count the number of conjugacy classes of maximal unipotent subgroups as follows. 

\begin{corollary}\label{coro number of max unip sub for GA}
There exists a bijective map between the set of conjugacy classes of maximal unipotent subgroups of $\mathbf{G}(\mathcal{O}_{\mathcal{S}})$ and $\mathrm{Pic}(\mathcal{O}_{\mathcal{S}})^{\mathbf{t}}$.
\end{corollary}

\begin{proof}
This is an immediate consequence of Theorem~\ref{main teo 1}~(3) and Theorem~\ref{main teo 3}.
\end{proof}

The following result shows that, when $\mathcal{O}_{\mathcal{S}}$ is a principal ideal domain, the characterization of the conjugacy classes of maximal unipotent subgroups of $\mathbf{G}(\mathcal{O}_{\mathcal{S}})$, is analogous to their description when we work over a field instead of $\OS$.

\begin{corollary}\label{coro max unip sub A principal}
The function ring $\mathcal{O}_{\mathcal{S}}$ is a principal ideal domain if and only if any maximal unipotent subgroup of $\mathbf{G}(\mathcal{O}_{\mathcal{S}})$ is conjugate to $\mathbf{U}^{+}(\mathcal{O}_{\mathcal{S}})$.
\end{corollary}

\begin{proof}
It is well known that $\mathcal{O}_{\mathcal{S}}$ is a principal ideal domain if and only if $\mathrm{Pic}(\mathcal{O}_{\mathcal{S}})$ is trivial.
But, Theorem~\ref{main teo 3} shows that $\mathrm{Pic}(\mathcal{O}_{\mathcal{S}})$ is trivial exactly when $\mathbf{G}(\mathcal{O}_{\mathcal{S}})$ acts transitively on $\SecD$.
Thus, $\mathcal{O}_{\mathcal{S}}$ is a principal ideal domain precisely when there exists a unique conjugacy class of maximal unipotent subgroups in $\mathbf{G}(\mathcal{O}_{\mathcal{S}})$.
Therefore, if each maximal unipotent subgroup of $\mathbf{G}(\mathcal{O}_{\mathcal{S}})$ is conjugate to $\mathbf{U}^{+}(\mathcal{O}_{\mathcal{S}})$, then $\mathcal{O}_{\mathcal{S}}$ is a principal ideal domain.

Conversely, if $\mathcal{O}_{\mathcal{S}}$ is principal, we can choose $\D_0$ as a representative of the unique $\mathbf{G}(\mathcal{O}_{\mathcal{S}})$-orbit.
Theorem~\ref{main teo 1}, applied with $h_\sigma=1$, shows that each unipotent subgroup of $\mathbf{G}(\mathcal{O}_{\mathcal{S}})$ is conjugate to $U(\mathbf{G}(\mathcal{O}_{\mathcal{S}}))= \mathbf{U}^{+}(k)  \cap \mathbf{G}(\mathcal{O}_{\mathcal{S}})=\mathbf{U}^{+}(\mathcal{O}_{\mathcal{S}})$.
\end{proof}

\begin{corollary}\label{coro max unip sub in sub of GA}
Assume that the field $\mathbb{F}$ is finite.
Then, each $\mathcal{S}$-arithmetic subgroup $G$ of $\mathbf{G}(k)$ has finitely many conjugacy classes of maximal unipotent subgroups.
Moreover, if $G \subseteq \mathbf{G}(\OS)$, then $G$ has at most $[\mathbf{G}(\mathcal{O}_{\mathcal{S}}):G] \cdot \card\big(\mathrm{Pic}(\mathcal{O}_{\mathcal{S}})\big)^{\mathbf{t}}$ conjugacy classes of maximal unipotent subgroups.
\end{corollary}

\begin{proof}
Assume that $G \subseteq \mathbf{G}(\OS)$.
Then, the number of $G$-orbits in $\SecD$ is less than or equal to the number of $\mathbf{G}(\mathcal{O}_{\mathcal{S}})$-orbits in $\SecD$ multiplied by the index $[\mathbf{G}(\mathcal{O}_{\mathcal{S}}):G]$.
Then, the second statement follows from Corollary~\ref{coro number of max unip sub for GA}.
Since $\mathrm{Pic}(\mathcal{O}_{\mathcal{S}})$ is finite whenever $\mathbb{F}$ is finite according to Weyl’s theorem (cf.~\cite[Ch. II, \S~2.2]{S}), the group $G \subseteq \mathbf{G}(\OS)$ has finitely many conjugacy classes of maximal unipotent subgroups.

Now, let $G$ be a arbitrary $\mathcal{S}$-arithmetic subgroup of $\mathbf{G}(k)$.
Let $G^{\natural}$ be a subgroup of $G \cap \mathbf{G}(\OS)$ such that the indices $[G:G^{\natural}]$ and $[\mathbf{G}(\mathcal{O}_{\mathcal{S}}): G^{\natural}]$ are finite. 
Since $[\mathbf{G}(\mathcal{O}_{\mathcal{S}}): G^{\natural}]$ is finite, the group $G^{\natural}$ has finitely many conjugacy classes of maximal unipotent subgroups.
Equivalently, the number of $G^{\natural}$-orbits in $\SecD$ is finite (cf. Theorem~\ref{main teo 1}~(3)).
Since $G^{\natural} \subseteq G$, the number of $G$-orbits in $\SecD$ is also finite.
Thus, it follows from Theorem~\ref{main teo 1}~(3) that $G$ has finitely many conjugacy classes of maximal unipotent subgroups. This proves the first statement.
\end{proof}

Let $\mathrm{Div}(\mathcal{C})$ be the group of divisor on the curve $\mathcal{C}$.
We denote by $\langle \mathcal{S} \rangle$ the subgroup of $\mathrm{Div}(\mathcal{C})$ generated by a finite set $\mathcal{S}$ of closed points of $\mathcal{C}$.
The following result shows that, for some large enough set of places $\mathcal{S}$, the group $\mathbf{G}(\OS)$ has a unique conjugacy class of maximal unipotent subgroups, since
$\OS$ is a principal ideal domain for such sets $\mathcal{S}$. This holds even when $\mathcal{C}$ has no rational points.

\begin{corollary}\label{coro unique max unip for big S}
Assume that the field $\mathbb{F}$ is finite.
Then, there exists a finite set of closed points $\mathcal{S}$, such that:
\begin{itemize}
\item[(C1)] each maximal unipotent subgroup of $\mathbf{G}(\mathcal{O}_{\mathcal{S}})$ is conjugate to $\mathbf{U}^{+}(\mathcal{O}_{\mathcal{S}})$.
\end{itemize}
Moreover, if $\mathcal{S}$ and $\mathcal{S'}$ satisfy (C1), $\mathcal{S} \cap \mathcal{S'} \neq \emptyset$ and $\langle \mathcal{S} \cup \mathcal{S}' \rangle$ has no non-trivial principal divisors, then $\mathcal{S} \cap \mathcal{S'}$ also satisfies (C1).
\end{corollary}

\begin{proof}
Let $\mathcal{S}_0=\lbrace \p \rbrace$, where $\p$ is a closed point of $\mathcal{C}$. Let us write $A=\mathcal{O}_{\mathcal{S}_0}$.
Since $\mathbb{F}$ is finite, the Picard group $\mathrm{Pic}(A)=\mathrm{Pic}(\mathcal{C})/\langle \overline{\p} \rangle$ is finite (cf.~\cite[Ch. II, \S 2.2]{S}).
Let $\p_1, \cdots, \p_n$ be a finite set of closed point of $\mathcal{C}$ such that $\overline{\p_1}, \cdots, \overline{\p_n}$ generate $\mathrm{Pic}(A)$.
Let us write $\mathcal{S}=\lbrace \p, \p_1, \cdots, \p_n \rbrace$.
Then, the Picard group $\mathrm{Pic}(\mathcal{O}_{\mathcal{S}})=\mathrm{Pic}(\mathcal{C})/\langle \overline{\mathcal{S}} \rangle$ is isomorphic to $\mathrm{Pic}(A) / \left( \langle \overline{\mathcal{S}} \rangle/ \langle \overline{\p} \rangle \right)$, which is trivial, since $\overline{\mathcal{S}}$ contains a set that generates $\mathrm{Pic}(A)$.
In other words, the ring $\mathcal{O}_{\mathcal{S}}$ is a principal ideal domain.
Thus, it follows from Corollary~\ref{coro max unip sub A principal} that $\mathcal{S}$ satisfies (C1). Hence, the first statement follows.

Now, assume that $\mathcal{S}$ and $\mathcal{S'}$ satisfy (C1) and $\mathcal{S} \cap \mathcal{S'} \neq \emptyset$. Then, it follows from Corollary~\ref{coro max unip sub A principal} that $\mathrm{Pic}(\mathcal{O}_{\mathcal{S}})$ and $\mathrm{Pic}(\mathcal{O}_{\mathcal{S}'})$ are trivial. We have the exact sequence
\begin{equation}\label{eq seq pic}
0 \to \left( \langle \overline{\mathcal{S}} \rangle/ \langle \overline{ \mathcal{S} \cap \mathcal{S}'} \rangle \right) \to \mathrm{Pic}(\mathcal{O}_{\mathcal{S} \cap \mathcal{S'}}) \to \mathrm{Pic}(\mathcal{O}_{\mathcal{S}}) \to 0.
\end{equation}
Let $\pi: \mathrm{Pic}(\mathcal{C}) \to \mathrm{Pic}(\mathcal{O}_{\mathcal{S}'})$ be the canonical
projection.
Then, the group ${\pi\left( \langle \overline{\mathcal{S}} \rangle \right)= \left( \langle \overline{\mathcal{S}} \rangle+\langle \overline{\mathcal{S}'} \rangle\right)/ \langle \overline{ \mathcal{S}'} \rangle} \cong \langle \overline{\mathcal{S}} \rangle/ \left(\langle \overline{ \mathcal{S}} \rangle \cap \langle \overline{ \mathcal{S}'} \rangle \right)$.
Since $\langle \mathcal{S} \cup \mathcal{S}' \rangle$ does not contains any non-trivial principal divisors, we have $\langle \overline{ \mathcal{S}} \rangle \cap \langle \overline{ \mathcal{S}'} \rangle= \langle \overline{ \mathcal{S} \cap \mathcal{S}'} \rangle$, whence $\pi\left( \langle \overline{\mathcal{S}} \rangle \right)= \left( \langle \overline{\mathcal{S}} \rangle/ \langle \overline{ \mathcal{S} \cap \mathcal{S}'} \rangle \right)$.
Hence, since $\mathrm{Pic}(\mathcal{O}_{\mathcal{S}'})$ is trivial, we have that $ \langle \overline{\mathcal{S}} \rangle/ \langle \overline{ \mathcal{S} \cap \mathcal{S}'} \rangle $ is trivial.
Therefore, since $\mathrm{Pic}(\mathcal{O}_{\mathcal{S}})$ is trivial, we conclude from Equation~\eqref{eq seq pic} that $\mathrm{Pic}(\mathcal{O}_{\mathcal{S} \cap \mathcal{S'}})$ is trivial, whence $\mathcal{S}\cap \mathcal{S}'$ satisfies (C1) according to Corollary~\ref{coro max unip sub A principal}.
\end{proof}

\section{Maximal unipotent subgroups}\label{section maximal unipotents}

The main goal of this section is to prove Theorem~\ref{main teo 1}.
We begin with some preliminary results.
The next lemma is straightforward.

\begin{lemma}\label{lemma indices for perfect fields}
Assume that $\mathbb{F}$ is a perfect field of characteristic $p>0$. Then ${[k:k^p]=p}$. \qed
\end{lemma}

\begin{lemma}\label{lemma unipotents in G}
Assume that the ground field $\mathbb{F}$ is perfect (of arbitrary characteristic).
Then, for each unipotent subgroup $U$ of $\mathbf{G}(k)$, there exists $h \in \mathbf{G}(k)$ such that ${U \subseteq h^{-1} \mathbf{U}^{+}(k) h}$.
\end{lemma}

\begin{proof}
Let $U$ be a unipotent subgroup of $\mathbf{G}(k)$. Since $\mathbf{G}_k$ is a simply connected semisimple $k$-group, the group $U$ is $k$-embeddable in the unipotent radical $\mathcal{R}_u(\mathbf{P}_k)$ of a $k$-parabolic subgroup $\mathbf{P}_k$ of $\mathbf{G}_k$.
This result is a consequence of \cite[Cor. 3.7]{BoTi2}, when $\mathbb{F}$ has characteristic zero, and of~\cite[Thm. 2]{G} and Lemma~\ref{lemma indices for perfect fields}, in positive characteristic.
In other words, if $\mathbb{F}$ is a perfect field, then $U$ is contained in $\mathcal{R}_u(\mathbf{P}_k)(k)$.
Moreover, since $\mathcal{R}_u(\mathbf{P}_k)$ is contained in the unipotent radical of some Borel subgroup $\mathbf{B}'_k$ of $\mathbf{G}_k$, and two Borel subgroups are $\mathbf{G}(k)$-conjugate, we conclude that $U \subseteq h^{-1} \mathbf{U}^{+}(k) h$, for some $h \in \mathbf{G}(k)$.
\end{proof}

Let $G \subset \mathbf{G}(k)$ be an $\mathcal{S}$-arithmetic subgroup of $\mathbf{G}(k)$.
For each $\alpha \in \Phi^{+}$, let ${\theta_{\alpha}: \mathbf{G}_a \to \mathbf{U}_{\alpha}}$ be the $\mathbb{Z}$-isomorphism given by the Chevalley pinning as in \S~\ref{section recall alg groups and BT}.
For each $h\in \mathbf{G}(k)$, we denote by $G_{\alpha}(h)$ and by $N_{\alpha}(h)$, the groups:
\begin{align*}
G_{\alpha}(h) :=& h^{-1} \mathbf{U}_{\alpha}(k) h \cap G,\\
N_{\alpha}(h):=&\theta_{\alpha}^{-1}\big(h G_{\alpha}(h) h^{-1}\big) = \theta_{\alpha}^{-1}\big( \mathbf{U}_{\alpha}(k) \cap h G h^{-1} \big).
\end{align*}
For $G=\mathbf{G}(\mathcal{O}_{\mathcal{S}})$ we write $M_{\alpha}(h):=N_{\alpha}(h)$.
\begin{lemma}\label{Lemma M_a}
For any $h \in \mathbf{G}(k)$ and any $\alpha \in \Phi^{+}$, the group $N_{\alpha}(h)$ is infinite.
\end{lemma}

\begin{proof}
By commensurability of $G$ with $\mathbf{G}(\mathcal{O}_{\mathcal{S}})$, there exists
 $G^{\natural} \subseteq \mathbf{G}(\mathcal{O}_{\mathcal{S}}) \cap G$ such that the indices $[G:G^{\natural}]$ and $[\mathbf{G}(\mathcal{O}_{\mathcal{S}}): G^{\natural}]$ are finite.
Let us write $M^{\natural}_{\alpha}(h):= \theta_{\alpha}^{-1}\big( \mathbf{U}_{\alpha}(k) \cap h G^{\natural} h^{-1} \big)$.
Let $\lbrace x_i \rbrace_{i \in \mathcal{I}}$ be a system of representatives of $M_{\alpha}(h)/ M^{\natural}_{\alpha}(h)$.
By definition of $M_{\alpha}(h)$, the element $g_i=h^{-1} \theta_{\alpha}(x_i)h$ belongs to $\mathbf{G}(\mathcal{O}_{\mathcal{S}})$.
If $g_i g_j^{-1}$ belongs to $G^{\natural}$, then $h^{-1}\theta_{\alpha}(x_i-x_j)h=g_i g_j^{-1} \in G^{\natural}$, or equivalently $x_i-x_j \in M^{\natural}_{\alpha}(h)$, whence $i=j$.
Thus $\big[M_{\alpha}(h): M^{\natural}_{\alpha}(h)\big] \leq \big[\mathbf{G}(\mathcal{O}_{\mathcal{S}}): G^{\natural}\big]$.
An analogous argument shows that $\big[N_{\alpha}(h): M^{\natural}_{\alpha}(h)\big] \leq \big[G: G^{\natural}\big].$
Thus, we have that $M_{\alpha}(h)$ and $N_{\alpha}(h)$ are commensurable.
Since $\mathcal{O}_{\mathcal{S}}$ is an integral domain, it follows from~\cite[Prop. 6.5]{BravoLoisel} that $M_\alpha(h)$ contains a non-zero $\mathcal{O}_{\mathcal{S}}$-ideal.
Thus, we conclude that $N_{\alpha}(h)$ is infinite.
\end{proof}

\begin{lemma}\label{lemma cont of unipotents}
Let $h_1, h_2 \in \mathbf{G}(k)$. Then, the group $h_1^{-1} \mathbf{U}^{+}(k) h_1 \cap G$ is contained in $h_2^{-1} \mathbf{U}^{+}(k) h_2 \cap G $ if and only if $h_2 h_1^{-1} \in \mathbf{B}(k)$.
In particular, if the equivalent conditions are satisfied, then
$$h_1^{-1} \mathbf{U}^{+}(k) h_1 \cap G = h_2^{-1} \mathbf{U}^{+}(k) h_2 \cap G.$$
\end{lemma}

\begin{proof}
Firstly, assume that $h_2 h_1^{-1} \in \mathbf{B}(k)$. Since $\mathbf{B}(k)$ normalizes $\mathbf{U}^{+}(k)$, we have $h_2 h_1^{-1} \mathbf{U}^{+}(k) h_1 h_2^{-1}  = \mathbf{U}^{+}(k)$. 
This implies that $h_2^{-1} \mathbf{U}^{+}(k) h_2 =h_1^{-1} \mathbf{U}^{+}(k) h_1$.
Thus, by intersecting both sides of the preceding equation with $G$, we get
$$h_1^{-1} \mathbf{U}^{+}(k) h_1 \cap G = h_2^{-1} \mathbf{U}^{+}(k) h_2 \cap G.$$

Now, we prove the converse implication. Let us write $\tau=h_2 h_1^{-1}$, so that $\tau (\mathbf{U}^{+}(k) \cap h_1 G h_1^{-1}) \tau^{-1}$ is contained in $\mathbf{U}^{+}(k) \cap h_2 G h_2^{-1}$.
For each $\alpha \in \Phi^{+}$, we can write $\mathbf{U}_{\alpha}(k) \cap h_1 G h_1^{-1}=\lbrace \theta_{\alpha}(z): z \in  N_{\alpha}(h_1) \rbrace$. Then, we have
\[ \tau \left\lbrace \prod_{\alpha \in \Phi^{+}} \theta_{\alpha}(x_{\alpha}) : x_{\alpha} \in N_{\alpha}(h_1) \right\rbrace \tau^{-1} \subseteq \mathbf{U}^{+}(k) \cap h_2 G h_2^{-1}.\]
By the Bruhat decomposition, let us write $\tau=b'wb$, where $b,b' \in \mathbf{B}(k)$ and $w \in N^{\mathrm{sph}}$, where $N^{\mathrm{sph}}$ is a lift of $W^{\mathrm{sph}} =\mathrm{N}_{\mathbf{G}(k)}(\mathbf{T}(k))/\mathbf{T}(k)$ to the group $\mathrm{N}_{\mathbf{G}(k)}(\mathbf{T}(k))$. Since $b'$ normalizes $\mathbf{U}^{+}(k)$, we deduce that:
\[ b \left( \prod_{\alpha \in \Phi^{+}} \theta_{\alpha}(x_{\alpha}) \right) b^{-1} \in \prod_{\alpha \in \Phi^{+}} \mathbf{U}_{w^{-1} \cdot \alpha}(k), \quad \forall (x_{\alpha})_{\alpha \in \Phi^+} \in  \prod_{\alpha \in \Phi^+} N_{\alpha}(h_1).\]
Assume, for the sake of contradiction, that $w \neq \mathrm{id}$.
Let $\alpha_0$ be a positive root such that the root $w^{-1} \cdot \alpha_0$ is negative.
Then, the $\alpha_0$-coordinate of $b \left( \prod_{\alpha \in \Phi^{+}} \theta_{\alpha}(x_{\alpha}) \right) b^{-1}$ equals zero, for all $(x_{\alpha})_{\alpha \in \Phi^{+}} \in \prod_{\alpha \in \Phi^{+}} N_{\alpha}(h_1)$. 
Since the linear algebraic $k$-group $\mathbf{B}$ normalizes the subgroup $\mathbf{U}^{+}$, the conjugacy map defined from $b \in \mathbf{B}(k)$ induces a scheme automorphism $\psi_b$ on $\mathbf{U}^{+}$. Moreover, since $\mathbf{U}^{+}$ has a parametrization (as $k$-variety)
$\prod_{\beta \in \Phi^{+}} \theta_\beta:  \prod_{\beta \in \Phi^{+}} \mathbb{G}_{a,k} \to \mathbf{U}^{+}$, the automorphism $\psi_b$ induces an automorphism $\phi_b$ on $\prod_{\beta \in \Phi^{+}} \mathbb{G}_{a,k}$. 
Let $\phi_b^{*}$ be the ring automorphism $\phi_b^{*}: \mathbb{F}[X_\beta | \beta \in \Phi^{+}] \to \mathbb{F}[X_\beta | \beta \in \Phi^{+}]$ induced by $\phi_b$.
For each $\alpha \in \Phi^{+}$, let us denote by $P_{\alpha}$ the polynomial in $\card\left(\Phi^{+}\right)$ indeterminates defined as the image of $X_\alpha$ by the ring automorphism $\phi_b^*$.
This polynomial depends on $b$ and on the ordering of factors of $\beta \in \Phi^{+}$.
Since each $X_{\alpha}$ is non-zero, we get that each polynomial $P_{\alpha}$ is non-zero.
Moreover, by definition of $P_\alpha$, we have that:
\[ b \left(\prod_{\alpha \in \Phi^{+}} \theta_{\alpha}(z_{\alpha}) \right) b^{-1}=\prod_{\alpha \in \Phi^{+}} \theta_{\alpha} \Big( P_{\alpha}\big((z_{\alpha})_{\alpha \in \Phi^{+}}\big) \Big) , \quad \forall(z_{\alpha})_{\alpha \in \Phi^{+}} \in \prod_{\alpha \in \Phi^{+}}\mathbb{G}_a. \]
Thus $ P_{\alpha_0}((x_{\alpha})_{\alpha \in \Phi^{+}})=0$, for all $(x_{\alpha})_{\alpha \in \Phi^{+}} \in \prod_{\alpha \in \Phi^{+}} N_{\alpha}(h_1)$.
In other words, we get that $P_{\alpha_0}$ vanishes on the product $\prod_{\alpha \in \Phi^{+}} N_{\alpha}(h_1)$, where any $N_{\alpha}(h_1)$ is infinite according to~Lemma~\ref{Lemma M_a}. Since $P_{\alpha_0} \neq 0$, we get a contradiction.
Therefore, we conclude that $w=\mathrm{id}$, whence $\tau=h_2h_1^{-1}$ belongs to $\mathbf{B}(k)$.
\end{proof}

We now turn to the proof of Theorem~\ref{main teo 1}.

\begin{proof}[Proof of Theorem~\ref{main teo 1}] 
We begin by proving Statement (1).
Indeed, let $u$ be a unipotent element of $G_{D}$. Since $u$ belongs to $h_D^{-1} \mathbf{B}(k) h_D$, there exists a unique pair $(\tilde{t}, \tilde{u}) \in h_D^{-1} \mathbf{T}(k) h_D \times h_D^{-1} \mathbf{U}^{+}(k) h_D $ such that $u= \tilde{t} \cdot \tilde{u}$, since a Borel subgroup is a semi-direct product of a maximal $k$-torus with its unipotent radical, according to~\cite[Prop. 21.13 (ii)]{BoA}.
Moreover, since $u$  is unipotent $\tilde{t}=\mathrm{id}$, whence $u \in h_D^{-1} \mathbf{U}^{+}(k) h_D$. 
Thus $u \in U(G_D)$.
Hence, we conclude that $U(G_{D})$ is the subgroup of all the unipotent elements in $G_D$. This proves Statement~(1).

Now, we prove Statement (2).
Indeed, let $U \subseteq G$ be a unipotent subgroup containing $U(G_{D})$.
It follows from Lemma~\ref{lemma unipotents in G} that $U$ is contained in $h^{-1} \mathbf{U}^{+}(k) h \cap G$, for some $h \in \mathbf{G}(k)$.
Hence
$$ h_D^{-1} \mathbf{U}^+(k) h_D \cap G = U(G_{D}) \subseteq U \subseteq h^{-1} \mathbf{U}^{+}(k) h \cap G.$$
Then, it follows from Lemma~\ref{lemma cont of unipotents} that the preceding inclusion becomes the equality
$$U(G_{D}) = U = h^{-1} \mathbf{U}^{+}(k) h \cap G.$$
We conclude that $U(G_{D})$ is a maximal unipotent subgroup of $G$.
Let $U$ be a maximal unipotent subgroup of $G$.
It follows from Lemma~\ref{lemma unipotents in G} that $U$ is contained in $h^{-1} \mathbf{U}^{+}(k) h \cap G$, for some $h \in \mathbf{G}(k)$.
We set $\D=h^{-1} \cdot \D_0$, so that $U(G_{D})$ equals $h^{-1} \mathbf{U}^{+}(k) h \cap G$.
In particular, we get $U \subseteq U(G_{D})$.
Moreover, since the unipotent group $U$ is assumed to be maximal, we conclude that $U=U(G_{D})$.
Thus, Statement~(2) follows.

It remains to prove Statement~(3).
Set $\mathfrak{U}/G=\lbrace U(G_{\sigma}) : \sigma \in \Sigma \rbrace$ as in Theorem~\ref{main teo 1}, and
let $U(G_{D})= h_D^{-1} \mathbf{U}^{+}(k) h_D \cap G$ be an element of $\mathfrak{U}$.
Then, by definition of the set $\lbrace \D_{\sigma} \rbrace_{\sigma \in \Sigma}$, there exists $\sigma \in \Sigma$ such that $\D=g \cdot \D_{\sigma}$, for some $g \in G$.
Therefore $U(G_{D})=g U(G_{\sigma}) g^{-1}$.
This proves that each $G$-conjugacy class in $\mathfrak{U}$ contains an element of $\mathfrak{U}/G$.

Conversely, let $\sigma_1,\sigma_2 \in \Sigma$ such that there exists an element $g \in G$ satisfying $ g U(G_{\sigma_1}) g^{-1}= U(G_{\sigma_2})$.
Let us write $D=g \cdot D_{\sigma_1}$.
Let $h$ be an element of $\mathbf{G}(k)$ such that $\D= h^{-1} \cdot \D_0$.
Then $U(G_D)= h^{-1} \mathbf{U}^{+}(k) h \cap G$.
Since $gh_{D_1}^{-1} \cdot \D_0=\D=h^{-1} \cdot \D_0$, we get $hgh_{D_1}^{-1} \in \mathrm{Stab}_{\mathbf{G}(k)}\left( \D_0 \right)$. Thus, it follows from Lemma~\ref{lemma stab of D_0} that $hgh_{D_1}^{-1}$ belongs to $\mathbf{B}(k)$.
Since $\mathbf{B}(k)$ normalizes $\mathbf{U}^{+}(k)$, we get:
\begin{equation}\label{eq U et h}
U(G_D) = h^{-1} \mathbf{U}^{+}(k) h \cap G= \big( g  h_{D_1}^{-1} \big) \mathbf{U}^{+}(k) \big( h_{D_1}g^{-1} \big) \cap G. 
\end{equation}
Moreover, since $g$ belongs to $G$, we have $g^{-1} G g=G$. Thus, Equality~\eqref{eq U et h} becomes
$$U(G_D) = g \left(  h_{D_1}^{-1} \mathbf{U}^{+}(k) h_{D_1} \cap G \right) g^{-1}=
g U(G_{\sigma_1}) g^{-1} =U(G_{\sigma_2}). $$
Therefore, Lemma~\ref{lemma cont of unipotents} applied to $U(G_D) = U(G_{\sigma_2})$ implies that $h_{D_2} h^{-1} \in \mathbf{B}(k)$,
whence we deduce that $\D= h^{-1} \cdot \D_0=h_{D_2}^{-1} \cdot \D_0=\D_{\sigma_2}$.
In other words, we conclude $g \cdot \D_{\sigma_1}=\D_{\sigma_2}$. This implies that $\sigma_1=\sigma_2$, by definition of the set of representatives $\lbrace \D_{\sigma} \rbrace_{\sigma \in \Sigma}$. Thus, Statement~(3) follows.
\end{proof}

\section{\texorpdfstring{$\mathbf{G}(\mathcal{O}_\mathcal{S})$}{G(OS)}-conjugacy classes of Borel \texorpdfstring{$k$}{k}-subgroups}\label{section prove of 2.2}

The main goal of this section is to prove Theorem~\ref{main teo 3}.
By definition of $\SecD$, for any $\D \in \SecD$, there exists $h \in \mathbf{G}(k)$ such that $\D=h^{-1} \cdot \D_0$.
Since $\mathrm{Stab}_{\mathbf{G}(k)}\left( \D_0 \right)=\mathbf{B}(k)$, there exists a $\mathbf{G}(k)$-equivariant bijection from $\SecD$ to the Borel variety $\mathbf{G}(k) /\mathbf{B}(k)$.
In particular, the set of $\mathbf{G}(\mathcal{O}_{\mathcal{S}})$-orbits in $\SecD$ is in bijection with the double quotient $\mathbf{G}(\mathcal{O}_{\mathcal{S}}) \backslash \mathbf{G}(k) /\mathbf{B}(k)$.

We begin this section by interpreting the double quotients $\mathbf{G}(\mathcal{O}_{\mathcal{S}}) \backslash \mathbf{G}(k) /\mathbf{B}(k)$ as the kernel of an homomorphism between two \'Etale cohomology groups.
This result is valid even when $\mathbf{G}_k$ is a quasi-split reductive $k$-group.

\begin{proposition}\label{prop number of cusps equals some kernell}
Let $\phi : H^1_{\text{\'et}}\big(\mathrm{Spec}(\mathcal{O}_{\mathcal{S}}),\mathbf{T}\big) \to H^1_{\text{\'et}}\big(\mathrm{Spec}(\mathcal{O}_{\mathcal{S}}),\mathbf{G}\big)$ be the map defined from the natural exact sequence $1 \to \mathbf{T} \to \mathbf{G} \to \mathbf{G}/\mathbf{T} \to 1$. There exists a bijective map from $\mathbf{G}(\mathcal{O}_{\mathcal{S}}) \backslash \mathbf{G}(k) /\mathbf{B}(k)$ to $\mathrm{ker}(\phi)$.
\end{proposition}

\begin{proof}
Firstly, we show that $(\mathbf{G}/\mathbf{B})(\mathcal{O}_{\mathcal{S}})= (\mathbf{G}/\mathbf{B})(k)$ (cf.~\cite[\S 4, Prop. 1.6]{Liu}).
Indeed, let $x \in (\mathbf{G}/\mathbf{B})(k)$ be a $k$-rational point of $\mathbf{G}/\mathbf{B}$.
By the valuative criterion of properness (cf.~\cite[\S 4, Prop. 1.6]{Liu}), for any prime ideal $\p'$ of $\mathcal{O}_{\mathcal{S}}$ there exists a unique element $x_{\p'} \in (\mathbf{G}/\mathbf{B})((\mathcal{O}_{\mathcal{S}})_{\p'})$ such that $x=\operatorname{Res}(x_{\p'})$.
Since the functor of points $\mathfrak{h}_V$ of $V=\mathbf{G}/\mathbf{B}$ is faithfully flat, we can assume that $x_{\p'} \in V((\mathcal{O}_{\mathcal{S}})_{f_{\p'}})$ where $f_{\p'} \notin \p'$.
Recall that $\operatorname{Spec}(\mathcal{O}_{\mathcal{S}})$ can be covered by a finite set $\left\lbrace\operatorname{Spec}((\mathcal{O}_{\mathcal{S}})_{f_{\p'_i}})\right\rbrace_{i=1}^n$.
Hence, by a patching argument, we find $\overline{x} \in V(\mathcal{O}_{\mathcal{S}})$ such that $x_{\p'_i}=\operatorname{Res}(\overline{x})$, and then $x=\operatorname{Res}(\overline{x})$.
This element is unique by local considerations. Thus, we have shown that $(\mathbf{G}/\mathbf{B})(\mathcal{O}_{\mathcal{S}})= (\mathbf{G}/\mathbf{B})(k)$.

Now, let us consider the following exact sequence of algebraic varieties:
$$
1 \rightarrow \mathbf{B} \xrightarrow{\iota} \mathbf{G} \xrightarrow{p} \mathbf{G}/\mathbf{B} \rightarrow 1.
$$
It follows from~\cite[Ch. III, \S 4, 4.6]{DG} that there exists a long exact sequence
\[
1 \to \mathbf{B}(\mathcal{O}_{\mathcal{S}}) \rightarrow \mathbf{G}(\mathcal{O}_{\mathcal{S}}) \rightarrow (\mathbf{G}/\mathbf{B})(\mathcal{O}_{\mathcal{S}}) \rightarrow H^1_{\text{\'et}}\big(\operatorname{Spec}(\mathcal{O}_{\mathcal{S}}), \mathbf{B}\big) \xrightarrow{} H^1_{\text{\'et}}\big(\operatorname{Spec}(\mathcal{O}_{\mathcal{S}}),\mathbf{G}\big).
\]
Moreover, it follows from~\cite[Exp.~XXVI, Cor.~2.3]{SGA3-3} that \[H^1_{\text{fppf}}\big(\operatorname{Spec}(\mathcal{O}_{\mathcal{S}}), \mathbf{B}\big)= H^1_{\text{fppf}}\big(\operatorname{Spec}(\mathcal{O}_{\mathcal{S}}), \mathbf{T}\big).\]
But, since $\mathbf{B}$ and $\mathbf{T}$ are both smooth over $\mathrm{Spec}(\mathcal{O}_{\mathcal{S}})$, we have
$$H^1_{\text{fppf}}\big(\operatorname{Spec}(\mathcal{O}_{\mathcal{S}}),\mathbf{B}\big) = H^1_{\text{\'et}}\big(\operatorname{Spec}(\mathcal{O}_{\mathcal{S}}),\mathbf{B}\big),$$
$$H^1_{\text{fppf}}\big(\operatorname{Spec}(\mathcal{O}_{\mathcal{S}}), \mathbf{T}\big)=H^1_{\text{\'et}}\big(\operatorname{Spec}(\mathcal{O}_{\mathcal{S}}), \mathbf{T}\big).$$
This implies that $H^1_{\text{\'et}}\big(\operatorname{Spec}(\mathcal{O}_{\mathcal{S}}), \mathbf{B}\big)= H^1_{\text{\'et}}\big(\operatorname{Spec}(\mathcal{O}_{\mathcal{S}}), \mathbf{T}\big)$. Thus, the previous long exact sequence becomes:
$$
1 \rightarrow \mathbf{B}(\mathcal{O}_{\mathcal{S}}) \rightarrow \mathbf{G}(\mathcal{O}_{\mathcal{S}}) \rightarrow (\mathbf{G}/\mathbf{B})(k) \rightarrow H^1_{\text{\'et}}\big(\operatorname{Spec}(\mathcal{O}_{\mathcal{S}}), \mathbf{T}\big) \rightarrow H^1_{\text{\'et}}\big(\operatorname{Spec}(\mathcal{O}_{\mathcal{S}}),\mathbf{G}\big).
$$
Then, it follows from~\cite[Ch. III, \S 4, 4.7]{DG} that there exists a bijection between $\mathbf{G}(\mathcal{O}_{\mathcal{S}}) \backslash (\mathbf{G}/\mathbf{B})(k)$ and the kernel:
$$\ker \left( H^1_{\text{\'et}}(\operatorname{Spec}(\mathcal{O}_{\mathcal{S}}), \mathbf{T}) \to H^1_{\text{\'et}}(\operatorname{Spec}(\mathcal{O}_{\mathcal{S}}),\mathbf{G})\right).$$
Therefore, since $(\mathbf{G}/\mathbf{B})(k) \cong \mathbf{G}(k)/\mathbf{B}(k)$ according to~\cite[\S 4.13(a)]{BoTi}, the result follows.
\end{proof}

\begin{corollary}\label{corollary cusps number in terms of pic}
Assume that $\mathbf{G}_k$ is a split simply connected semisimple $k$-group.
Let $\mathbf{t}$ be the rank of $\mathbf{G}_k$.
Then, there is a one-to-one correspondence between $\mathbf{G}(\mathcal{O}_{\mathcal{S}}) \backslash \mathbf{G}(k) /\mathbf{B}(k)$ and $\operatorname{Pic}(\mathcal{O}_{\mathcal{S}})^\mathfrak{t}$.
\end{corollary}

\begin{proof}
It follows from Hilbert's Theorem~90 (cf.~\cite[Ch. III, Prop. 4.9]{Mi}) that: 
\[H^1_{\text{Zar}}\big(\operatorname{Spec}(\mathcal{O}_{\mathcal{S}}), \mathbb{G}_m\big) = H^1_{\text{\'et}}\big(\operatorname{Spec}(\mathcal{O}_{\mathcal{S}}), \mathbb{G}_m\big) \cong \operatorname{Pic}(\mathcal{O}_{\mathcal{S}}).\]
Since $\mathbf{T}$ is split over $\mathbb{Z}$, we have that $\mathbf{T} \cong \mathbb{G}_{m,\mathbb{Z}}^{\mathbf{t}}$. 
Thus, we get:
$$H^1_{\text{Zar}}\big(\operatorname{Spec}(\mathcal{O}_{\mathcal{S}}), \mathbf{T}\big)= H^1_{\text{\'et}}\big(\operatorname{Spec}(\mathcal{O}_{\mathcal{S}}), \mathbf{T}\big) \cong \operatorname{Pic}(\mathcal{O}_{\mathcal{S}})^\mathbf{t}.$$
Since $\phi \Big(H^1_{\text{Zar}}(\operatorname{Spec}(\mathcal{O}_{\mathcal{S}}), \mathbf{T})\Big) \subseteq H^1_{\text{Zar}}\big(\operatorname{Spec}(\mathcal{O}_{\mathcal{S}}), \mathbf{G}\big)$ and $H^1_{\text{Zar}}\big(\operatorname{Spec}(\mathcal{O}_{\mathcal{S}}), \mathbf{T}\big)$ equals $ H^1_{\text{\'et}}\big(\operatorname{Spec}(\mathcal{O}_{\mathcal{S}}), \mathbf{T}\big)$, we have:
\[
\mathrm{ker}(\phi) 
=\ker \Big( H^1_{\text{Zar}}\big(\operatorname{Spec}(\mathcal{O}_{\mathcal{S}}), \mathbf{T}\big)  \to H^1_{\text{Zar}}\big(\operatorname{Spec}(\mathcal{O}_{\mathcal{S}}),\mathbf{G}\big)\Big).
\]
Moreover, since $\mathbf{G}_k$ is a simply connected semisimple $k$-group and $\mathcal{O}_{\mathcal{S}}$ is a Dedekind domain, it follows from~\cite[Thm. 2.2.1 \& Cor. 2.3.2]{H1} that $H^1_{\text{Zar}}(\operatorname{Spec}(\mathcal{O}_{\mathcal{S}}),\mathbf{G})$ is trivial, for any integral model $\mathbf{G}$ of $\mathbf{G}_k$.
We conclude that
$$\ker(\phi) = H^1_{\text{Zar}}\big(\operatorname{Spec}(\mathcal{O}_{\mathcal{S}}),\mathbf{T}\big) \cong  \operatorname{Pic}(\mathcal{O}_{\mathcal{S}})^\mathbf{t},$$
whence the result follows.
\end{proof}

\section{Fixing groups of germs at infinity of chambers}\label{section fixing}

In order to prove Theorem~\ref{main teo 2} we have to decompose the semisimple group $G_D/U(G_D)$ as the product of a free $\mathbb{Z}$-module of finite rank and a ``arithmetic bounded torus'' $T$, which is a torsion group whenever $\mathbb{F}$ is finite.
In this section, we describe the group $T$ in terms of the stabilizer of a certain filter, which we precise in the following.

The sector chambers of $\mathbb{A}_0$ with direction $D_0$ form a basis of a filter $\gamma(D_0)$, according to~\cite[\S 7.2.3]{BT1}.
The germ associated to this basis is called the germ of $\mathbb{A}_0$ with direction $D_0 = \left( D_{0,\p} \right)_{\p \in \mathcal{S}}$.
For each $\p \in \mathcal{S}$, it induces a germ of $\mathbb{A}_{0,\p}$ with direction $D_{0,\p}$.
Since the set of $\mathbf{G}(k_\p)$-pointwise stabilizer subgroups of those filters is a directed set according to~\cite[\S 7.2.2]{BT1},
the union of these pointwise stabilizers over a basis of the filter $\gamma(D_{0,\p})$ forms a group.
The same holds for pointwise stabilizers of those filters in any subgroup of $\mathbf{G}(k_\p)$.
Thus, the union of the pointwise stabilizers in $\mathbf{G}(k_p)$ of the sector chambers with direction $D_{0,\p}$ is a group, which we call the pointwise stabilizer in $\mathbf{G}(k_\p)$ of the filter $\gamma(D_{0,\p})$.
We denote it by:
\[\widehat{P}^\p_{\gamma(D_{0,\p})}:=
\bigcup_{x \in \mathbb{A}_{0,\p}} \fix_{\mathbf{G}(k_\p)}\big(Q(x_\p,D_{0,\p})\big).
\]
This group decomposes as:
\begin{equation}\label{eq decomposition of stab of germ}
\widehat{P}^\p_{\gamma(D_{0,\p})}=\mathbf{T}(\mathcal{O}_\p) \mathbf{U}^+(k_\p),
\end{equation}
whenever $\mathbf{G}$ is semisimple~\cite[\S 7.2.3]{BT1} and split\footnote{The groups $H = \mathbf{T}(k_\p)_b$, $H^0$ and $H^1 = \mathbf{T}(\mathcal{O}_\p)$ defined in \cite[\S 4.6.3]{BT2} are the same since $\mathbf{T}$ is a smooth connected integral model over $\mathcal{O}_\p$ of a maximal (split) torus as being defined over $\mathbb{Z}$ (see~\cite[\S 4.6.32]{BT2}).}.

Analogously, given an apartment $\mathbb{A}$ and a vector chamber $D$, we define a filter $\gamma(D)$ given by the basis of sector chambers of $\mathbb{A}$ with direction $D$.
The pointwise stabilizer of $\gamma(D)$ is defined as the union of the pointwise stabilizers of sector chambers in $\mathbb{A}$ with direction $D$.

Let $\D \in \SecD$ and $h_D \in \mathbf{G}(k)$ such that $\D=h_D^{-1} \cdot \D_0$.
Let $\mathbb{A}$ be the apartment of $\X$ defined as 
\begin{equation}\label{eq appatment}
\mathbb{A}=h_D^{-1} \cdot \mathbb{A}_0.
\end{equation}
We want to consider the pointwise stabilizer in $G$ of the filter $\gamma(D) = \big(\gamma(D_\p)\big)_{\p \in \mathcal{S}}$ generated by the germ of $\mathbb{A}$ with direction $D$.

For each $\p \in \mathcal{S}$, the pointwise stabiliser of $\gamma(D_\p)$ in $G$ is 
\[\bigcup_{x_\p \in \mathbb{A}_\p} \fix_G\big( Q(x_\p,D_\p) \big) = h_D^{-1} \Big( \bigcup_{y_\p \in \mathbb{A}_{0,\p}}\fix_{\mathbf{G}(k_\p)}\big(Q(y_\p,D_{0,\p})\big) \cap h_D G h_D^{-1} \Big) h_D.\]
Thus, the pointwise stabiliser of $\gamma(D)$ in $G$ is
\[
G_{\gamma(D)}:=\bigcap_{\p \in \mathcal{S}} h_D^{-1} \Big( \bigcup_{y_\p \in \mathbb{A}_{0,\p}}\fix_{\mathbf{G}(k_\p)}\big(Q(y_\p,D_{0,\p})\big) \cap h_D G h_D^{-1} \Big) h_D.
\]
Since $\mathcal{S}$ is finite, we have
\[ G_{\gamma(D)} = h_D^{-1} \Big( \bigcup_{y = (y_\p)_{\p \in \mathcal{S}} \in \mathbb{A}_{0}} \bigcap_{\p \in \mathcal{S}} \fix_{\mathbf{G}(k_\p)}\big(Q(y_\p,D_{0,\p})\big) \cap h_D G h_D^{-1} \Big) h_D.
\]
Moreover, since 
\[\bigcap_{\p \in \mathcal{S}} \Big( \fix_{\mathbf{G}(k_\p)}\big(Q(y_\p,D_{0,\p})\big) \cap h_D G h_D^{-1} \Big) = \fix_{h_D G h_D^{-1}} \Big(Q\big((y_\p)_{\p \in \mathcal{S}},D_0\big)\Big),\]
we deduce that 

\begin{equation}\label{eq G germ of D}
G_{\gamma(D)} = \bigcup_{y \in \mathbb{A}} \fix_G\big(Q(y,D)\big).  
\end{equation}

For any $y\in \mathbb{A}$, the pointwise stabilizer $\mathrm{Fix}_{G}(Q(y,D))$ is contained in $G_D=\mathrm{Stab}_G(\D)$.
Thus, the group $G_{\gamma(D)}$ is contained in $G_D$, which equals $h_D^{-1} \mathbf{B}(k) h_D \cap G$.
Using the semi-direct product decomposition of a Borel subgroup into a maximal torus and its unipotent radical, for any $g \in G_{\gamma(D)}$ there exists a unique pair $(t,u) \in \big( h_D^{-1} \mathbf{T}(k) h_D \big) \times \big( h_D^{-1} \mathbf{U}^{+}(k) h_D \big)$ such that $g=t \cdot u$.
Then, we define:
\begin{equation}\label{eq T germ}
T \left(G_{\gamma(D)} \right):= \left \lbrace t \in h_D^{-1} \mathbf{T}(k) h_D: \exists u \in h_D^{-1} \mathbf{U}^{+}(k) h_D, \, \,  t \cdot u \in G_{\gamma(D)} \right\rbrace,
\end{equation}
and
\begin{equation}\label{eq U germ}
U \left(G_{\gamma(D)} \right):= h_D^{-1} \mathbf{U}^{+}(k) h_D \cap G_{\gamma(D)}.
\end{equation}
In \S~\ref{section proof of 2.3} we prove that the ``arithmetic bounded torus'' $T$ introduced in Theorem~\ref{main teo 3} is exactly the group $h_D T \left(G_{\gamma(D)} \right) h_D^{-1}$.
The following result allows us to describe each group $T \left(G_{\gamma(D)} \right)$ as a quotient of $G_{\gamma(D)}$.

\begin{lemma}\label{lemma torus-unip}
The map $f: G_{\gamma(D)} \to h_D^{-1} \mathbf{T}(k) h_D$ defined by $f(tu)=t$ is a group homomorphism and it induces a group isomorphism $G_{\gamma(D)}/ U\left( G_{\gamma(D)} \right) \cong T\left( G_{\gamma(D)} \right)$.
\end{lemma}

\begin{proof}
Let $g_1= t_1 \cdot u_1$ and $g_2= t_2 \cdot u_2$, as above. 
Since $h_{D}^{-1} \mathbf{T}(k) h_{D}$ normalizes $h_{D}^{-1} \mathbf{U}^{+}(k) h_{D}$, we have $g_1 g_2^{-1}= t_1 u_1 u_2^{-1} t_2^{-1}=t_1 t_2^{-1} \tilde{ u}_1 \tilde{u}_2$, for some $\tilde{u}_1, \tilde{u}_2 \in h_{D}^{-1} \mathbf{U}^{+}(k) h_{D}$.
Thus, by uniqueness of the writing, we deduce that $f$ is a group homomorphism.
Note that, by definition, $\ker(f)=h_D^{-1} \mathbf{U}^{+}(k) h_D \cap G_{\gamma(D)}= U \left( G_{\gamma(D)} \right)$, and $\mathrm{Im}(f)=T \left( G_{\gamma(D)} \right)$. Thus, the result follows.
\end{proof}

The following lemmas allow us to describe the structure of the unipotent subgroup $U\left( G_{\gamma(D)} \right)$ of $G_{\gamma(D)}$.

\begin{lemma}\label{lemma unipotent fixes subsectors0}
For any $u \in \mathbf{U}^{+}(k)$
there exists a $k$-sector chamber $Q_{\p}=Q(y_{\p},D_{0,\p})$ in $\mathbb{A}_{0,\p}$ such that $u \in \mathrm{Fix}_{\mathbf{G}(k)}(Q_{\p})$.
\end{lemma}

\begin{proof}
Let us write $u\in \mathbf{U}^{+}(k)$ as $u=\prod_{\alpha \in \Phi^{+}} \theta_{\alpha}(x_{\alpha})$, for some $(x_{\alpha})_{\alpha \in \Phi^{+}} \in \prod_{\alpha \in \Phi^{+}} k$.
Let $v$ be a vertex in the standard apartment $\mathbb{A}_{0,\p}$ of $X(\mathbf{G},k,\p)$ defined in \S~\ref{section recall alg groups and BT}. According to~\cite[\S 6.4.9]{BT1}, the group $\prod_{\alpha \in \Phi} \theta_{\alpha}\left(\pi_{\p}^{-\alpha(v)} \mathcal{O}_{\p}\right)$ is contained in $\mathrm{Stab}_{\mathbf{G}(k_{\p})}(v)$.
In particular, the unipotent element $u$ is contained in $\mathrm{Stab}_{\mathbf{G}(k_{\p})}(v)$, for any vertex $v \in \mathbb{A}_{0,\p}$ such that $\alpha(v) \geq -\nu_{\p}(x_{\alpha})$.
This implies that $u$ fixes the complex 
$$E_{\p}:= \left\lbrace z \in \mathbb{A}_{0,\p} :  \alpha(z) > -\nu_{\p}(x_{\alpha}), \forall \alpha \in \Phi^{+} \right\rbrace.$$
Let $y_{\p} \in E_{\p}$. Then $Q(y_{\p}, D_{0,\p})$ is contained in $E_{\p}$. Thus, we conclude that $u$ fixes $Q(y_{\p}, D_{0,\p})$, as wished.
\end{proof}

\begin{lemma}\label{lemma unipotent fixes subsectors}
Let $\D \in \SecD$.
For any $g \in h_D^{-1}\mathbf{U}^{+}(k) h_D$
there exists a $k$-sector chamber $Q=Q(y,D)$, with $y \in \mathbb{A}=h_D^{-1} \cdot \mathbb{A}_0$, such that $g \in \mathrm{Fix}_{\mathbf{G}(k)}(Q)$.
\end{lemma}

\begin{proof}
Let $u = h_D g h_D^{-1} \in \mathbf{U}^{+}(k)$. We have to show that there exists $Q'=Q(y', D_0)$ with $y' \in \mathbb{A}_0$ such that $u \in \mathrm{Fix}_{\mathbf{G}(k)}(Q')$.
Indeed, it follows from Lemma~\ref{lemma unipotent fixes subsectors0} that there exists a sector chamber $Q_{\p}=Q(y_{\p},D_{0,\p})$ in $\mathbb{A}_{0,\p}$ such that $u \in \mathrm{Fix}_{\mathbf{G}(k)}(Q_{\p})$.
Let $Q'=\prod_{\p \in \mathcal{S}} Q_{\p}$ be a $k$-sector chamber of $\X$, which is contained in $\mathbb{A}_{0}$ by definition of $\mathbb{A}_0$ (cf.~\S~\ref{section recall alg groups and BT}).
Then, the direction of $Q'$ is $D_0$.
Moreover, since for any point $z=(z_{\p})_{\p \in \mathcal{S}}$ we have $\mathrm{Stab}_{\mathbf{G}(k)} (z)=\bigcap_{\p \in \mathcal{S}} \mathrm{Stab}_{\mathbf{G}(k)} (z_{\p})$, we get
$$\mathrm{Fix}_{\mathbf{G}(k)}(Q')= \bigcap_{\p \in \mathcal{S}}\mathrm{Fix}_{\mathbf{G}(k)}(Q_{\p}), $$
whence we conclude that $u$ belongs to $\mathrm{Fix}_{\mathbf{G}(k)}(Q')$.
Therefore, if we set $Q=h_D^{-1} \cdot Q'$, then we have $Q=Q(h_D^{-1}\cdot y', h_D^{-1} \cdot D_0)= Q(y,D)$, for $x=h_D^{-1}\cdot y' \in \mathbb{A}$, and $g \in \mathrm{Fix}_{\mathbf{G}(k)}(Q)$.
\end{proof}

\begin{proposition}\label{prop equal unuipotent part}
For any $\D \in \SecD$, we have 
$$U \left( G_{\gamma(D)} \right)= h_D^{-1}\mathbf{U}^{+}(k) h_D \cap G =U(G_D).$$
Moreover, $U \left( G_{\gamma(D)} \right)$ is the subgroup of all unipotent elements in $G_{\gamma(D)}$.
\end{proposition}

\begin{proof}
Firstly,
it follows from its definition that the group $U \left( G_{\gamma(D)} \right)$ is contained in $h_D^{-1}\mathbf{U}^{+}(k) h_D \cap G$.
Conversely, let $u \in h_D^{-1}\mathbf{U}^{+}(k) h_D \cap G$ be an arbitrary element.
It follows from Lemma~\ref{lemma unipotent fixes subsectors} that $u$ fixes a $k$-sector chamber $Q(y,D)$, where $y \in \mathbb{A}$.
Thus, it follows from Equation~\eqref{eq G germ of D} that $g$ belongs to $G_{\gamma(D)}$.
Hence, we conclude that $u$ belongs to $h_D^{-1}\mathbf{U}^{+}(k) h_D \cap G_{\gamma(D)}$, which equals $U \left( G_{\gamma(D)} \right)$.

Let $u$ be a unipotent element of $G_{\gamma(D)}$.
Since $u$ is a unipotent element of the Borel subgroup $h_D^{-1} \mathbf{B}(k) h_D$, it is contained in its unipotent radical $h_D^{-1} \mathbf{U}^+(k) h_D$ (since it is defined over $k$ according to~\cite[Thm.~15.4]{BoA}).
Thus $u \in U \left( G_{\gamma(D)} \right)$.
Since $U \left( G_{\gamma(D)} \right)$ is unipotent, we conclude that $U \left( G_{\gamma(D)}\right)$ is the subgroup of all the unipotent elements in $G_{\gamma(D)}$.
\end{proof}

We now focus on the description of the semisimple group $T \left(G_{\gamma(D)} \right)$.

\begin{proposition}\label{prop torus in F}
Assume that $G \subseteq \mathbf{G}(\OS)$.
Then, for any $\D \in \SecD$, we have
$T \left( G_{\gamma(D)} \right) \subseteq h_D^{-1} \mathbf{T}(\mathbb{F}) h_D$. 
\end{proposition}

\begin{proof}
Let $t \in T \left( G_{\gamma(D)} \right)$ be an arbitrary element.
By definition of $T \left( G_{\gamma(D)} \right)$, there is $u \in h_D^{-1} \mathbf{U}^+(k) h_D$ such that $tu \in G_{\gamma(D)}$.
Moreover, by definition of $G_{\gamma(D)}$, the element $tu$ pointwise stabilizes a sector chamber $Q=Q(y,D)$ of $\mathbb{A}$.
In particular, $t u$ stabilizes a vertex $z \in Q$. 
Let us write $z=(z_{\p})_{\p \in \mathcal{S}}$, where $z_{\p} \in X(\mathbf{G},k,\p)$.
Since $\mathbf{G}(k)$ acts diagonally on $\X$, the element $tu$ stabilizes each coordinate $z_{\p}$ of $z$.

Consider the element $b=h_D \left( tu \right) h_D^{-1} \in \mathbf{G}(k)$.
Since $t \in h_D^{-1} \mathbf{T}(k) h_D$ and $u \in h_D^{-1} \mathbf{U}^{+}(k) h_D$, we have that $b \in \mathbf{B}(k)$.
Moreover, since $tu$ stabilizes each point $z_{\p}$, the element $b$ belongs to the parahoric subgroup $\widehat{P}_{z'_{\p}} = \stab_{\mathbf{G}(k_{\p})}(z'_{\p}) \subseteq \mathbf{G}(k_{\p})$, where $z'_{\p}=h_D \cdot z_{\p}$.
In other words, if we write $z'=h_D \cdot z$, then $b$ belongs to $\bigcap_{\p \in \mathcal{S}} \left( \widehat{P}_{z'_{\p}} \cap \mathbf{B}(k)\right).$
In the Borel subgroup, write $b=sv$ with
$s \in \mathbf{T}(k)$ and $v \in \mathbf{U}^{+}(k)$.
By uniqueness of the pair $(s,v)$ defined above, we have $s = h_D t h_D^{-1}$ and $v=h_D uh_D^{-1}$.

Recall that there is a canonical isomorphism $X^*(\mathbf{T}) = X^*(\mathbf{B})$ since $\mathcal{R}_u(\mathbf{B}) = \mathbf{U}^+$ and $\mathbf{B} = \mathbf{T} \ltimes \mathbf{U}^+$.
In particular, for any $\chi \in X^*(\mathbf{B})$, we have that $\chi(b) = \chi(s)$.
According to~\cite[\S 8.1]{BT1}, the parahoric subgroups of $\mathbf{G}(k)$ are bounded\footnote{In fact, the parahoric subgroups are compact whenever $\mathbb{F}$ is finite, since $\mathbf{G}(k)$ acts properly on its building $\mathcal{X}_k$ whenever $\mathbf{G}$ is semisimple and $k$ is locally compact.}.
Hence, so is $\widehat{P}_{z'_{\p}} \cap \mathbf{B}(k) \subset \mathbf{G}(k_{\p})$.
Since $\mathbf{G}_k$ is semisimple, for each $\p \in \mathcal{S}$, the set of values $\nu_\p\Big( \chi \left( \widehat{P}_{z'_{\p}} \cap \mathbf{B}(k) \right) \Big)$ is bounded below\footnote{We are in the case where $\mathbf{G} = \mathbf{G}^1$ since $\mathbf{G}_k$ is assumed to be semisimple (cf.~\cite[\S 4.2.16]{BT2}).} according to~\cite[\S 4.2.19]{BT2}.
Thus, $\chi\left( \widehat{P}_{z'_{\p}} \cap \mathbf{B}(k) \right)$ is a subgroup of $\mathcal{O}_{\p}^{\times}$. Hence, we conclude 
$$ \chi\left( \widehat{P}_{z'} \cap \mathbf{B}(k) \right) = \chi \left( \bigcap_{\p \in \mathcal{S}} \widehat{P}_{z'_{\p}} \cap \mathbf{B}(k) \right) \subseteq \bigcap_{\p \in \mathcal{S}} \chi \left( \widehat{P}_{z'_{\p}} \cap \mathbf{B}(k) \right) \subseteq \bigcap_{\p \in \mathcal{S}} \mathcal{O}_{\p}^{\times}.$$

Up to conjugation by an element in $\mathrm{GL}_n(k)$, there is a faithful linear representation $\rho : \mathbf{G} \to \mathrm{GL}_{n,k}$ embedding $\mathbf{T}$ in the maximal torus of $\mathrm{GL}_{n,k}$ consisting in the group of diagonal matrices $\mathrm{D}_n$ and embedding $\mathbf{B}$ in the standard Borel subgroup $\mathrm{B}_n$ of $\mathrm{GL}_{n,k}$ consisting of upper triangular matrices.
Thus, the restriction of $\rho$ to the injective group homomorphism $\mathbf{T} \to \mathrm{D}_n$ induces a surjective homomorphism $\rho^* : X^*(\mathrm{D}_n) \to X^*(\mathbf{T})$ (cf.~\cite[\S 1.2]{BoTi}). 
For any character $\chi' \in X^*(\mathrm{B}_n) = X^*(\mathrm{D}_n)$, we have $\rho^*(\chi') \in X^*(\mathbf{T}) = X^*(\mathbf{B})$, whence $\chi'\big(\rho(b)\big) \in \bigcap_{\p \in \mathcal{S}} \mathcal{O}_{\p}^{\times}$.
Thus, the eigenvalues of $\rho(b)$ belong to $\bigcap_{\p \in \mathcal{S}} \mathcal{O}_{\p}^{\times}$.

Since $G \subseteq \mathbf{G}(\OS)$, the element $b$ belongs to $h_D \mathbf{G}(\mathcal{O}_{\mathcal{S}}) h_D^{-1}$.
Then, $\rho(b)$ is conjugate in $\mathrm{GL}_n(k)$ to a matrix in $\mathrm{GL}_n(\mathcal{O}_{\mathcal{S}})$.
Thus, the characteristic polynomial $P_{\rho(b)}$ of $\rho(b)$ has coefficients in $\mathcal{O}_{\mathcal{S}}$.
Since the eigenvalues are in $\bigcap_{\p \in \mathcal{S}} \mathcal{O}_{\p}$, we deduce that $P_{\rho(b)}$ has coefficients in $\mathcal{O}_{\mathcal{S}}\cap \bigcap_{\p \in \mathcal{S}} \mathcal{O}_{\p} = \mathbb{F}$.
Since $\mathcal{C}$ is assumed geometrically integral over $\mathbb{F}$, the field $\mathbb{F}$ is algebraically closed in $k$ (cf.~\S~\ref{section intro}).
Thus, $P_{\rho(b)}$ splits over $k$.
We deduce that the eigenvalues of $\rho(b)$ are in $\mathbb{F}^{\times}$, whence $\chi'\big(\rho(b)\big) \in \mathbb{F}^{\times}$ for any $\chi' \in X^*(\mathrm{B}_n)$.

Since $\rho^*$ is surjective, we deduce that $\chi(s)=\chi(b) \in \mathbb{F}^\times$ for all  $\chi \in X^*(\mathbf{B})$.
Thus, since $\mathbf{T}$ is split and defined over $\mathbb{Z}$, we deduce from the perfect dual pairing (cf.~\cite[\S 8.6]{BoA}) that $s \in \mathbf{T}(\mathbb{F})$.
Hence $t = h_D^{-1} s h_D $ belongs to $ h_D^{-1} \mathbf{T}(\mathbb{F}) h_D$.
\end{proof}

\begin{corollary}\label{coro T in F}
For any $\mathcal{S}$-arithmetic subgroup $G$ of $\mathbf{G}(k)$, the group $T\left( G_{\gamma(D)} \right)$ is commensurable with a subgroup of $h_D^{-1}\mathbf{T}(\mathbb{F}) h_D$.
\end{corollary}

\begin{proof}
Since $G$ is commensurable with $\mathbf{G}(\OS)$, there exists a subgroup $G^{\natural}$ of $G \cap \mathbf{G}(\OS)$ such that $[G:G^{\natural}]$ and $[\mathbf{G}(\mathcal{O}_{\mathcal{S}}): G^{\natural}]$ are finite.
Since $G^{\natural} \subseteq \mathbf{G}(\OS)$, it follows from Proposition~\ref{prop torus in F} that $T \left( G^{\natural}_{\gamma(D)} \right) \subseteq h_D^{-1} \mathbf{T}(\mathbb{F}) h_D$.
Thus, in order to prove the result, we have to check that $T \left( G^{\natural}_{\gamma(D)} \right)$ has finite indexed in $T \left( G_{\gamma(D)} \right)$.
It follows from Equation~\eqref{eq G germ of D} that $G_{\gamma(D)}$ equals
$$\bigcup_{y \in \mathbb{A}} \fix_G\big(Q(y,D)\big)= \bigcup_{y \in \mathbb{A}} \left( G \cap \fix_{\mathbf{G}(k)}\big(Q(y,D)\big) \right)  =G \cap \left( \bigcup_{y \in \mathbb{A}} \fix_{\mathbf{G}(k)}\big(Q(y,D)\big)\right).$$
In other words, the group $G_{\gamma(D)}$ is the intersection of $G$ with the group $\mathbf{G}(k)_{\gamma(D)}:=\bigcup_{y \in \mathbb{A}} \fix_{\mathbf{G}(k)}\big(Q(y,D)\big)$.
An analogous argument shows that $G^{\natural}_{\gamma(D)}$ is the intersection of $G^{\natural}$ with $\mathbf{G}(k)_{\gamma(D)}$.
Since, for any $H \subseteq \mathbf{G}(k)$ we have $[G \cap H: G^{\natural} \cap H] \leq [G:G^{\natural}]$, we get, by taking $H=\mathbf{G}(k)_{\gamma(D)}$, that $G^{\natural}_{\gamma(D)}$ has finite index in $G_{\gamma(D)}$.
Moreover, since $G^{\natural}_{\gamma(D)} \subseteq G_{\gamma(D)}$, it follows from Equation~\eqref{eq U germ} that $U \left(G^{\natural}_{\gamma(D)} \right)$ equals
$$ 
h_D^{-1} \mathbf{U}^{+}(k) h_D \cap G^{\natural}_{\gamma(D)}= \left( h_D^{-1} \mathbf{U}^{+}(k) h_D \cap G_{\gamma(D)} \right) \cap G^{\natural}_{\gamma(D)} = U \left(G_{\gamma(D)} \right) \cap G^{\natural}_{\gamma(D)}.
$$
In other words, we have $U \left(G^{\natural}_{\gamma(D)} \right)= U \left(G_{\gamma(D)} \right) \cap G^{\natural}_{\gamma(D)}$.
Since, for any normal subgroup $N$ of $G_{\gamma(D)}$ we have $[G_{\gamma(D)}/N: G_{\gamma(D)}^{\sharp}/(N \cap G^\sharp_{\gamma(D)})] \leq [G_{\gamma(D)}:G^{\sharp}_{\gamma(D)}]$, we obtain, by taking $N=U \left(G_{\gamma(D)} \right)$, that $G^{\natural}_{\gamma(D)}/ U \left( G^{\natural}_{\gamma(D)} \right)$ has finite indexed in $G_{\gamma(D)}/U \left( G_{\gamma(D)} \right)$.
Thus, the result follows from Lemma~\ref{lemma torus-unip}.
\end{proof}

In the context where $\mathcal{S}=\lbrace \p \rbrace$ and $G \subseteq \mathbf{G}(\OS)$, the $G$-stabilizer of the germ of $\mathbb{A}$ with direction $D$ is the $G$-stabilizer of $\D$. More precisely, we have the following result.

\begin{proposition}\label{coro stab of sector chamber for p}
Assume that $\mathcal{S}=\lbrace \p \rbrace$ and $G \subseteq \mathbf{G}(\OS)$. Then, for any $\D \in \SecD$, we have 
$$G_{\gamma(D)}= h_D^{-1} \mathbf{B}(k) h_D \cap G =G_D.$$
\end{proposition}

\begin{proof}
Let $\D$ be an arbitrary element in $\SecD$.
Note that, since $G \subseteq \mathbf{G}(\mathcal{O}_{\lbrace \p \rbrace})$, if the result holds for $G=\mathbf{G}(\OS)$, i.e., if $\mathbf{G}(\mathcal{O}_{\lbrace \p \rbrace})_{\gamma(D)}= h_D^{-1} \mathbf{B}(k) h_D \cap \mathbf{G}(\mathcal{O}_{\lbrace \p \rbrace}) =\mathbf{G}(\mathcal{O}_{\lbrace \p \rbrace})_D$, then the result holds for any subgroup $G \subseteq \mathbf{G}(\OS)$, since $G_{\gamma(D)}= h_D^{-1} \mathbf{B}(k) h_D \cap G =G_D$.
Thus, without loss of generality, we may assume that $G=\mathbf{G}(\mathcal{O}_{\lbrace \p \rbrace})$.

Note that, for any $y\in \mathbb{A}$, the pointwise stabilizer $\mathrm{Fix}_{\mathbf{G}(\mathcal{O}_{\lbrace \p \rbrace})}(Q(y,D))$ is contained in $\mathrm{Stab}_{\mathbf{G}(\mathcal{O}_{\lbrace \p \rbrace})}(\D)$.
Therefore, the group $\mathbf{G}(\mathcal{O}_{\lbrace \p \rbrace})_{\gamma(D)}$ is contained in $\mathbf{G}(\mathcal{O}_{\lbrace \p \rbrace})_D$, which equals $h_D^{-1} \mathbf{B}(k) h_D \cap \mathbf{G}(\mathcal{O}_{\lbrace \p \rbrace})$.
Hence, it remains to prove $\mathbf{G}(\mathcal{O}_{\lbrace \p \rbrace})_D \subseteq \mathbf{G}(\mathcal{O}_{\lbrace \p \rbrace})_{\gamma(D)}$.
Let $Q= h_D^{-1} \cdot Q(x_0,D_0)$ be a sector chamber in $\mathbb{A}$ with direction $D$.
In particular, the sector chamber $Q$ has the form $Q(x,D)$, for certain $x \in \mathbb{A}$.
It follows from~\cite[Thm.~2.2]{BravoLoisel} that there exists a sector chamber $Q'=Q(x',D) \subseteq Q$ which embeds in the quotient $\mathbf{G}(\mathcal{O}_{\lbrace \p \rbrace}) \backslash \mathcal{X}(\mathbf{G},k, \p)$.
Let $g \in \mathbf{G}(\mathcal{O}_{\lbrace \p \rbrace})_D$.
The complex $Q' \cap g \cdot Q'$  is the intersection of two sector chamber with same direction.
In particular, it follows from~\cite[Lemma~11.77]{AB} that $Q' \cap g \cdot Q'$ contains a sector chamber $Q''$, which has the form $Q''=Q(x'',D)$.
Since $Q'' \subseteq Q'$, we have $Q'' \subseteq \mathbb{A}$, whence $x'' \in \mathbb{A}$. 
Let $z$ be a point in $Q''$.
Then $z$ belongs to $Q' \cap g \cdot Q'$, whence there exists $w \in Q'$ satisfying $z= g \cdot w$.
Moreover, since $g$ belongs to $\mathbf{G}(\mathcal{O}_{\lbrace \p \rbrace})$ and since $Q'$ does not have two points in the same $\mathbf{G}(\mathcal{O}_{\lbrace \p \rbrace})$-orbit, we have that $z=w$.
Thus, we conclude that $g \in \mathrm{Stab}_{\mathbf{G}(\mathcal{O}_{\lbrace \p \rbrace})}(z)$, for all $z \in Q''$.
Hence $g$ belongs to $\mathbf{G}(\mathcal{O}_{\lbrace \p \rbrace})_{\gamma(D)}$, which completes the proof.
\end{proof}

\section{Comparison between the stabilizers of a vector chamber and of its germ}\label{section proof of 2.3}

This section is dedicated to prove Theorem~\ref{main teo 3}.
In order to do this, in this section we compare the $G$-stabilizers $G_D$ of a vector chamber $D$ with the $G$-stabilizer $G_{\gamma(D)}$ of its germ $\gamma(D)$.

As above, let $\D \in \SecD$ and $h_D \in \mathbf{G}(k)$ be such that $h_D \cdot \D = \D_0$. 
Recall that the $G$-stabilizer $G_D$ of the vector chamber $D$, or equivalently of $\D$, equals $h_D^{-1} \mathbf{B}(k) h_D \cap G$.
We define a diagonalisable group by
\begin{equation}\label{eq def of T_s}
T(G_{D})=\left\lbrace t \in h_D^{-1} \mathbf{T}(k) h_D: \exists u \in h_D^{-1} \mathbf{U}^{+}(k) h_D,\ tu \in G_{D} \right\rbrace.
\end{equation}

This group is a quotient of $G_D$.
Indeed, by using the semi-direct product decomposition of the Borel subgroup $h_D^{-1} \mathbf{B}(k) h_D$ into a maximal torus and its unipotent radical, any element $g \in G_D$ can be written uniquely $g=tu$ where $t \in h_D^{-1} \mathbf{T}(k) h_D$ and $u \in h_D^{-1} \mathbf{U}^+(k) h_D$. Thus, one can define a map:
\begin{equation}\label{eq def of f}
\begin{array}{cccc}
f:& G_{D}& \to& h_D^{-1} \mathbf{T}(k) h_D\\
&g=tu& \mapsto &t
\end{array}
\end{equation}
which is, in fact, a group homomorphism.
Moreover, the kernel of $f$ is $U(G_{D})$ and its image is $T\left( G_D \right)$.
The next result follows.

\begin{lemma}\label{lemma torus-unip 2}
One has $G_{D}/U(G_{D}) \cong T(G_{D})$. \qed
\end{lemma}

\begin{lemma}\label{G finite is normal in G}
Let $G_{\gamma(D)}$ be the group defined in Equation~\eqref{eq G germ of D}.
Then, $G_{\gamma(D)}$ is a normal subgroup in $G_{D}$.
\end{lemma}

\begin{proof}
Let $\tau \in G_{D}$ and let $g \in G_{\gamma(D)}$ be arbitrary elements. 
Then, by definition of $G_{\gamma(D)}$, the element $g$ belongs to $\mathrm{Fix}_G(Q)$, for some sector chamber $Q$ contained in $\mathbb{A}$.
Since, by its definition, $G_{D}$ equals $\mathrm{Stab}_G(\D)$, we have $\tau \cdot \D=\D$.
Thus, it follows from~\cite[Lemma~11.77]{AB} that the intersection $Q \cap \tau \cdot Q$ contains a sector chamber $Q'$.
In particular, since $Q' \subseteq Q$, the direction of the sector chamber $Q'$ is $D$ and $Q' \subset \mathbb{A}$.
Since $\tau \in G$, the element $\tau g \tau^{-1}$ belongs to $\mathrm{Fix}_G(\tau \cdot Q)$.
But, since $Q' \subseteq \tau \cdot Q$, we have $\mathrm{Fix}_G(\tau \cdot Q) \subseteq \mathrm{Fix}_G(Q')$.
Thus, we conclude that $\tau g \tau^{-1}$ belongs to $\mathrm{Fix}_G(Q')$, where $Q' \subseteq \mathbb{A}$.
This implies that $\tau g \tau^{-1} \in G_{\gamma(D)}$, whence the result follows.
\end{proof}

It follows from Lemma~\ref{G finite is normal in G} that the quotient groups $G_{D}/G_{\gamma(D)}$ and $T(G_{D})/T \left( G_{\gamma(D)} \right)$ make sense. The following proposition describe their structure.
As in \S~\ref{section max unip sub}, we denote by $\mathbf{t}$ the rank of $\mathbf{G}$.

 \begin{proposition}\label{proposition comp of G_s and Gf}
There exists an integer $r=r(\mathbf{G}, \mathcal{S},G,D) \in \mathbb{Z}_{\geq 0}$, with $r \leq \mathbf{t} \cdot \sharp \mathcal{S}$, such that $G_{D}/G_{\gamma(D)}\cong T(G_{D})/T \left( G_{\gamma(D)}\right)\cong \mathbb{Z}^{r}$.
 \end{proposition}

\begin{proof}
At first we prove that $G_{D}/G_{\gamma(D)}$ and $T(G_{D})/T \left( G_{\gamma(D)}\right)$ are isomorphic.
 Indeed, we have $G_{D}/U(G_{D}) \cong T(G_{D})$ according to Lemma~\ref{lemma torus-unip 2} and that $G_{\gamma(D)}/U \left( G_{\gamma(D)} \right) \cong T \left( G_{\gamma(D)}\right)$ according to Lemma~\ref{lemma torus-unip}. 
 Then, since Proposition~\ref{prop equal unuipotent part} shows that $U \left( G_{\gamma(D)} \right)=h_{D}^{-1} \mathbf{U}^{+}(k) h_{D} \cap G= U(G_{D})$, we conclude that $G_{D}/G_{\gamma(D)} \cong T(G_{D})/ T \left( G_{\gamma(D)} \right)$.

Now, we prove that $G_{D}/G_{\gamma(D)}\cong \mathbb{Z}^{r}$, for some $0 \leq r \leq \mathbf{t} \cdot \sharp \mathcal{S}$.
Indeed, let $h \in \mathbf{G}(k)$ be such that $\D = h^{-1} \cdot \D_0$.
Then, by Equality~\eqref{eq decomposition of stab of germ}, we have that the pointwise stabilizer of $\gamma(D)$ in $\widehat{G}_\mathcal{S}$ is $\prod_{\p \in \mathcal{S}} h \big(\mathbf{T}(\mathcal{O}_\p) \cdot \mathbf{U}^+(k_\p) \big) h^{-1}$.
Thus, by the diagonal action of $G$ on $\X$, the pointwise stabilizer $G_{\gamma(D)}$ in $G$ of $\gamma(D)$ can be written as:
\[
    G \cap \bigcap_{\p\in\mathcal{S}} \mathbf{G}(k) \cap \left( h \mathbf{T}(\mathcal{O}_\p) \cdot \mathbf{U}^+(k_\p)h^{-1} \right)
    = G \cap h \left( \bigcap_{\p \in \mathcal{S}} \mathbf{G}(k) \cap \left( \mathbf{T}(\mathcal{O}_\p) \cdot \mathbf{U}^+(k_\p) \right) \right) h^{-1}.\]
Since $\mathbf{T}(\mathcal{O}_\p) \cdot \mathbf{U}^+(k_\p) \subset \mathbf{B}(k_\p)$ and $\mathbf{G}(k) \cap \mathbf{B}(k_\p) = \mathbf{B}(k)$, in $\mathbf{G}(k_\p)$, we have that
\[\mathbf{G}(k) \cap \left( \mathbf{T}(\mathcal{O}_\p) \cdot \mathbf{U}^+(k_\p) \right) =\mathbf{B}(k) \cap \left(\mathbf{T}(\mathcal{O}_\p) \cdot \mathbf{U}^+(k_\p) \right).\]
Consider the quotient group 
\[\Lambda_h := h \mathbf{B}(k) h^{-1} / h \left( \bigcap_{\p \in \mathcal{S}} \mathbf{B}(k) \cap \left( \mathbf{T}(\mathcal{O}_\p) \cdot \mathbf{U}^+(k_\p) \right) \right) h^{-1}\]
By restricting the quotient homomorphism $\pi_h:h\mathbf{B}(k)h^{-1} \to \Lambda_h$ to $G_D$, we deduce that $G_D / G_{\gamma(D)}$ is isomorphic to a certain subgroup of $\Lambda_h$.

Consider the diagonal group homomorphism
\[\varphi: \mathbf{B}(k) \to \prod_{\p \in \mathcal{S}} \mathbf{B}(k_\p) / \left( \mathbf{T}(\mathcal{O}_\p) \cdot \mathbf{U}^+(k_\p)\right).\]
For any $\p \in \mathcal{S}$, we have that $\mathbf{B}(k_\p) / \left( \mathbf{T}(\mathcal{O}_\p) \cdot \mathbf{U}^+(k_\p) \right) \cong \mathbf{T}(k_\p) / \mathbf{T}(\mathcal{O}_\p) \cong \mathbb{Z}^{\mathbf{t}}$, where $\mathbf{t}$ is the dimension of $\mathbf{T} \cong \mathbb{G}_m^{\mathfrak{t}}$, since $\mathbf{G}$ is split.
Moreover, since $\ker \varphi = \bigcap_{\p \in \mathcal{S}} \mathbf{B}(k) \cap \left( \mathbf{T}(\mathcal{O}_\p) \cdot \mathbf{U}^+(k_\p) \right)$, we deduce that $\Lambda_h$ is isomorphic to a subgroup of $\operatorname{im}\varphi = \left(\mathbb{Z}^\mathbf{t}\right)^{\mathcal{S}}$.
Hence $\Lambda_h$ is a finitely generated free $\mathbb{Z}$-module, whence so is $G_D / G_{\gamma(D)}$ as a submodule of $\Lambda_h$.
We denote by $r(\mathbf{G},\mathcal{S},G,D)$ its rank as free $\mathbb{Z}$-module, which is less than or equal to $\mathbf{t} \cdot \sharp \mathcal{S}$ by construction.
\end{proof}

\begin{proof}[End of the proof of Theorem~\ref{main teo 2}]
Statements~(1) and~(4) directly follow from Lemma~\ref{lemma torus-unip 2} and Proposition~\ref{coro stab of sector chamber for p}, respectively.
Since $T(G_{D})$ and $T \left( G_{\gamma(D)}\right)$ are abelian groups, it follows from Proposition~\ref{proposition comp of G_s and Gf} that there exists an exact sequence
\begin{equation}\label{exact sequence in torus}
0 \to T \left( G_{\gamma(D)}\right) \to T(G_{D}) \to \mathbb{Z}^r \to 0,  
\end{equation}
where $r \leq \mathbf{t} \cdot \sharp \mathcal{S}$.
Since $\mathbb{Z}^r$ is $\mathbb{Z}$-free, the exact sequence~\eqref{exact sequence in torus} splits. 
Thus, the group $T(G_{D})$ is isomorphic to a semi-direct product of $T \left( G_{\gamma(D)}\right)$ by $\mathbb{Z}^r$. 
Note that, since $T(G_{D})$ is abelian, it is isomorphic to the direct product of $T \left( G_{\gamma(D)}\right)$ and $\mathbb{Z}^r$.
Thus, Statement~(2) follows from Corollary~\ref{coro T in F} by setting $T:=h_D T \left( G_{\gamma(D)}\right) h_D^{-1}$.
Moreover, Statement~(3) is a direct consequence of Proposition~\ref{prop torus in F}.

Now, assume that $G \subseteq \mathbf{G}(\OS)$, $\mathcal{S}=\lbrace \p \rbrace$, $\mathbb{F}$ is finite of characteristic $p$ and that the torsion of $G$ is $p$-primary.
Since $G \subseteq \mathbf{G}(\OS)$ and $\mathcal{S}=\lbrace \p \rbrace$, it follows from Proposition~\ref{coro stab of sector chamber for p} that $G_{D}=G_{\gamma(D)}$. 
Let $g \in G_{D}=G_{\gamma(D)}$ be an arbitrary element.
We write it as $g = t \cdot u$, where $t \in T \left( G_{\gamma(D)} \right)$ and $u \in U \left( G_{\gamma(D)}\right)$ (cf. Lemma~\ref{lemma torus-unip}).
Since $\mathbb{F}$ is finite, the group $T \left( G_{\gamma(D)}\right)$ has cardinality $q$ coprime with $p$, according to Proposition~\ref{prop torus in F}.
Then $g^q=t^q \cdot \tilde{u}=\tilde{u}$, for certain $\tilde{u} \in h^{-1}_{D} \mathbf{U}^{+}(k) h_{D}$.
Recall that, up to conjugate by an element in $\mathrm{GL}_{n,k}(k)$, there exists a faithful $k$-linear representation $\mathbf{G}(k) \to \mathrm{GL}_{n,k}(k)$ embedding $\mathbf{U}^{+}(k)$ in the group of unipotent upper triangular matrices $\mathrm{U}_n(k)$ of $\mathrm{GL}_{n,k}(k)$.
Since $\mathrm{Char}(\mathbb{F})=p$, $\mathrm{U}_n(k)$ is a torsion group, whence $\mathbf{U}^{+}(k)$ also is a torsion group.
In particular, the unipotent element $\tilde{u}\in h_D^{-1} \mathbf{U}^{+}(k) h_{D}$ has finite order.
We deduce that $g$ has finite order.
Thus, since the torsion of $G$ is $p$-primary, the order of $g$ equals $p^f$, for some $f \in \mathbb{Z}_{\geq 0}$.
Let $a,b \in \mathbb{Z}$ be such that $1=ap^f+bq$.
Then $g= (g^{p^f})^a \cdot (g^q)^b=(g^q)^b \in  h_D^{-1} \mathbf{U}^{+}(k) h_D$. Hence, we deduce that $G_{D} \subseteq h^{-1}_{D} \mathbf{U}^{+}(k) h_{D} \cap G=U(G_{D})$, whence Statement~(5) follows.
\end{proof}

Recall that the isomorphism classes of line bundles over $\mathcal{C}$ form the Picard group $\mathrm{Pic}(\mathcal{C})$ with the tensor product as composition law.
If $e$ denotes the gcd of the degrees of closed points on $\mathcal{C}$, then we have the exact sequence:
$$ 0 \to \mathrm{Pic}^{0}(\mathcal{C}) \to \mathrm{Pic}(\mathcal{C}) \xrightarrow{\deg} e \mathbb{Z} \to 0. $$
The group $\mathrm{Pic}^{0}(\mathcal{C})$ is called the Jacobian variety of $\mathcal{C}$. This is finite when $\mathbb{F}$ is finite according to Weyl's theorem (cf.~\cite[Ch. II, \S~2.2]{S}).

\begin{example}\label{ex for D_0}
Let $\D_0 \in \SecD$ as in \S~\ref{section diagonal action}. In order to simplify our calculations, we set $h_{D_0}=\mathrm{id}$.
Let $G=\mathbf{G}(\OS)$ be the group of $\OS$-points of $\mathbf{G}$. 
By definition, we have
\begin{equation}\label{eq G_D and U(G_D) for D_0}
G_{D_0}= \mathbf{B}(k) \cap \mathbf{G}(\OS)= \mathbf{B}(\OS) \text{ and } U \left( G_{D_0} \right) = \mathbf{U}^{+}(k) \cap \mathbf{G}(\OS)= \mathbf{U}^{+}(\OS).   
\end{equation}
Recall that $T(G_{D_0})$ is the image of $G_{D_0}$ by the group homomorphism $f: G_{D_0} \to \mathbf{T}(k)$ defined by $f(g)=t$, where $g=t \cdot u$.
Moreover, since $\ker(f)=U \left( G_{D_0} \right)$ according to Lemma~\ref{lemma torus-unip 2}, it follows from Equation~\eqref{eq G_D and U(G_D) for D_0} that 
\begin{equation}\label{eq T(G_D) for D_0}
T\left(G_{D_0}\right) = \mathbf{T}(\OS).
\end{equation}
Now, we decompose $T(G_{D_0})$ as a direct product, as in Theorem~\ref{main teo 2}~(2).
To get such a decomposition, consider $T \left(G_{\gamma(D_0)} \right)$ the semisimple group defined in Equation~\eqref{eq T germ}.
On the one hand, Proposition~\ref{prop torus in F} shows that $T \left( G_{\gamma(D_0)} \right)$ is a subgroup of $\mathbf{T}(\mathbb{F})$.
On the other hand, since $\nu_{\p}(\mathbb{F}^{*})= \lbrace 0 \rbrace$, for any place $\p$ on $k$, we have that $\mathbf{T}(\mathbb{F})$ fixes any point in $\mathbb{A}$.
This implies that $\mathbf{T}(\mathbb{F})$ is contained in each $\mathrm{Fix}_G\left(Q(y,D)\right)$, for $y \in \mathbb{A}_0$. In other words, we have $\mathbf{T}(\mathbb{F}) \subseteq G_{\gamma(D_0)}$. 
Thus, we get that $\mathbf{T}(\mathbb{F}) \subseteq T  \left( G_{\gamma(D_0)} \right)$, whence:
\begin{equation}\label{eq T(g(G_D)) for D_0}
 T  \left( G_{\gamma(D_0)} \right)=\mathbf{T}(\mathbb{F}).
 \end{equation}
Thus, it follows from Equation~\eqref{eq T(G_D) for D_0} and Equation~\eqref{eq T(g(G_D)) for D_0} that $T\left(G_{D_0}\right) / T\left(G_{\gamma(D_0)}\right)$ equals $ \mathbf{T}(\OS)/\mathbf{T}(\mathbb{F}).$
Since $\mathbf{T} \cong \mathbb{G}_m^{\mathbf{t}}$, where $\mathbf{t}=\mathrm{rk}(\mathbf{G})$, the following diagram commutes:
$$
\def\commutatif{\ar@{}[rd]|{\circlearrowleft}}
\xymatrix{
\mathbf{T}(\mathbb{F}) \ar[r]^{\cong} \ar@{^{(}->}[d]^{\iota} \commutatif & \left( \mathbb{F}^{*} \right)^{\mathbf{t}} \ar@{^{(}->}[d]^{\iota} 
\\
\mathbf{T}(\OS) \ar[r]^{\cong} & \left( \OS^{*} \right)^{\mathbf{t}}  
 }.$$
In other words, we have 
\begin{equation}
T\left(G_{D_0}\right) / T\left(G_{\gamma(D_0)}\right) = \mathbf{T}(\OS)/\mathbf{T}(\mathbb{F}) \cong \left( \OS^{*} \right)^{\mathbf{t}} / \left( \mathbb{F}^{*} \right)^{\mathbf{t}} \cong \left( \OS^{*}/ \mathbb{F}^{*}\right) ^{\mathbf{t}}
\end{equation}
It follows from the Dirichlet unit theorem (cf.~\cite[\S 14, Cor. 1]{Ro}) that $\OS^{*}$ is isomorphic to the direct product of $\mathbb{F}^{*}$ and a $\mathbb{Z}$-free group $\Lambda'$, whose rank is at most $\sharp S-1$.
Moreover, \cite[\S 14, Prop. 14.2]{Ro} shows that the rank of $\Lambda'$ is exactly $\sharp S-1$ when $\mathrm{Pic}^{0}(\mathcal{C})$ is a torsion group.
This is the case when $\mathbb{F}$ is finite.
We conclude that $T\left( G_{D_0} \right)$ is isomorphic to the direct product of $\mathbf{T}(\mathbb{F})$ and the $\mathbb{Z}$-free group $\Lambda:=\left( \OS^{*}/ \mathbb{F}^{*}\right) ^{\mathbf{t}} \cong (\Lambda')^{\mathbf{t}}$, whose rank $r_{D_0}$ is at most $\mathbf{t} \cdot (\sharp S-1)$, with equality when $\mathrm{Pic}^{0}(\mathcal{C})$ is a torsion group.
\end{example}

\begin{remark}
Assume that $\mathbb{F}$ is an algebraic (possible infinite) extension of a finite field $\mathbb{F}_p$.
Since any element of $\mathrm{Pic}^0(\mathcal{C})$ is defined over a finite extension $\mathbb{L} \subset \mathbb{F}$ of $\mathbb{F}_p$, the group $\mathrm{Pic}^0(\mathcal{C})$ is a torsion group. 
In particular $r_{D_0}= \mathrm{rk}(\mathbf{G}) \cdot (\sharp \mathcal{S}-1)$, for such fields.

Recall that the Jacobian variety $\mathrm{Pic}^0(\mathcal{C})$ is an Abelian variety.
Moreover, recall that the group of $n$-torsion points $\mathcal{A}[n]$ of any abelian variery $\mathcal{A}$ is a finite group.
This implies that the group of torsion points of $\mathcal{A}$ is countable.
In particular, the group of torsion points of $\mathrm{Pic}^0(\mathcal{C})$ is countable.
Let $\mathbb{F}$ be a non-countable perfect field.
We can take, for instance, $\mathbb{F}=\mathbb{C}$, in characteristic $0$, and $\mathbb{F}=\overline{\mathbb{F}_{p}((T))}$, in characteristic $p$.
Then, for an elliptic curve $\mathcal{E}$, the Jacobian variety $\mathrm{Pic}^0(\mathcal{E}) \cong \mathcal{E}(\mathbb{F})$ is non-countable.
This implies that $\mathrm{Pic}^0(\mathcal{E})$ has non-torsion elements.
In particular, the integer $r_{D_0}$ is strictly smaller than $\mathrm{rk}(\mathbf{G}) \cdot (\sharp \mathcal{S}-1)$, for such curves.
\end{remark}

It follows from Theorem~\ref{main teo 2}~(2)-(3) that, for each $\D \in \SecD$, the semisimple group $T((\mathbf{G}(\OS))_{D})$ is isomorphic to the direct product of a subgroup of $\mathbf{T}(\mathbb{F})$ with $\mathbb{Z}^{r_{0,D}}$, for some $r_{0,D}:=r(\mathbf{G}, \mathcal{S},D) \in \mathbb{Z}$, which just depends on the triple $(\mathbf{G}, \mathcal{S},D)$.
Moreover, when $D=D_0$, Example~\ref{ex for D_0} shows that $r_{0,D}=r_{D_0} \leq \mathbf{t}(\sharp \mathcal{S}-1)$, with equality whenever $\mathrm{Pic}^0(\mathcal{C})$ is a torsion group.
In the remaining part of this section, we find a certain suitable decomposition for $T(G_D)$ in terms of $r_{0,D}$, in the case where $G$ is a normal finite index subgroup of $\mathbf{G}(\OS)$.

\begin{lemma}\label{lemma G normal and its subgroups}
Assume that $G$ is a normal subgroup of $\mathbf{G}(\OS)$.
Let $\D' \in \SecD$ be a chamber which belongs to the $\mathbf{G}(\OS)$-orbit of $\D \in \SecD$.
Then:
\begin{itemize}
\item[(1)] $G_{D'}$ is $\mathbf{G}(\mathcal{O}_{\mathcal{S}})$-conjugate to $G_D$,
\item[(2)] $U(G_{D'})$ is $\mathbf{G}(\mathcal{O}_{\mathcal{S}})$-conjugate to $U(G_D)$, and
\item[(3)] $T(G_{D'})$ is isomorphic to the direct product of a subgroup of $\mathbf{T}(\mathbb{F})$ and a $\mathbb{Z}$-free group, whose rank is exactly $r_{0,D}$. 
\end{itemize}
\end{lemma}

\begin{proof}
By definition of $\D'$, there exists $g \in \mathbf{G}(\OS)$ such that $\D'= g^{-1} \cdot \D$.
In particular, we have $h_{D'}= h_D g$.
Since $G$ is normal in $\mathbf{G}(\mathcal{O}_{\mathcal{S}})$, we get 
$$h_{D'}^{-1} \mathbf{B}(k) h_{D'} \cap  G= (gh_D)^{-1}  \mathbf{B}(k) (g h_D)\cap  G = g^{-1} \left( h_{D}^{-1} \mathbf{B}(k) h_{D} \cap  G\right) g.$$
In other words, we have 
\begin{equation}\label{G_D is g-conjugated}
G_{D'}= g^{-1} G_D g.
\end{equation}
By an analogous argument, we get
$$h_{D'}^{-1} \mathbf{U}^{+}(k) h_{D'} \cap  G= (gh_D)^{-1}  \mathbf{U}^{+}(k) (g h_D)\cap  G = g^{-1} \left( h_{D}^{-1} \mathbf{U}^{+}(k) h_{D} \cap  G\right) g,$$
whence
\begin{equation}\label{U(G_D) is g-conjugated}
U\left(G_{D'}\right)= g^{-1} U \left( G_D \right) g.
\end{equation}
Thus, Statement~(1) and~(2) follow.

Now, we prove Statement~(3). Let $\mathbb{A}= h_D^{-1} \cdot D_0$ and $\mathbb{A'}= h_{D'}^{-1} \cdot D_0$ as in Equation~\eqref{eq appatment}.
Since $h_{D'}= h_D g$, we have $\mathbb{A}'=g^{-1} \cdot \mathbb{A}$. 
In particular, for any $z \in \mathbb{A}'$, the point $y=g^{-1} \cdot y$ belongs to $\mathbb{A}$.
Since $G$ is normal in $\mathbf{G}(\OS)$, we have
\begin{multline*}
\fix_G\big(Q(z,D)\big)= \fix_{\mathbf{G}(k)}\big(Q(z,D)\big) \cap G =  g^{-1} \fix_{\mathbf{G}(k)}\big(Q(y,D')\big) g \cap G
\\ 
= g^{-1} \left( \fix_{\mathbf{G}(k)}\big(Q(y,D')\big)  \cap  g G g^{-1} \right) g
= g^{-1} \fix_G\big(Q(y,D)\big) g.
\end{multline*}
Thus, it follows from Equation~\eqref{eq G germ of D} that $G_{\gamma(D')} \subseteq g^{-1} G_{\gamma(D)} g$. 
An analogous argument provides the converse inclusion, whence 
\begin{equation}\label{eq G(gamma(D)) is g-conjugated}
G_{\gamma(D')} = g^{-1} G_{\gamma(D)} g.   
\end{equation}
Hence, it follows from Proposition~\ref{proposition comp of G_s and Gf} that
\begin{equation}\label{eq T/T are isomorphics}
T(G_D)/ T(G_{\gamma(D)}) \cong G_D/ G_{\gamma(D)} \cong  G_{D'}/ G_{\gamma(D')} \cong T(G_{D'})/ T(G_{\gamma(D')}). 
\end{equation}
Since $G \subseteq \mathbf{G}(\OS)$, it follows from Theorem~\ref{main teo 2}~(2)-(3) that $T(G_D)$ (resp. $T(G_{D'})$) is isomorphic to the direct product of a subgroup of $\mathbf{T}(\mathbb{F})$ and a free $\mathbb{Z}$-module, whose rank we denote here as $r(D,G)$ (resp. $r(D',G)$).
In particular, it follows from Equation~\eqref{eq T/T are isomorphics} that $r(D,G)=r(D',G)$.
In other words, the map $D' \mapsto r(D',G)$ is a constant function on the $G$-orbit of $D$.
Thus, in order to prove Statement~(3) we have to show that $r(D,G)$ equals $r_{0,D}$.

Since $G_{\gamma(D)}= \left( \mathbf{G}(\OS)\right)_{\gamma(D)} \cap G$, the natural inclusion $G_D \to \left( \mathbf{G}(\OS)\right)_{D}$ induces an injective group homomorphism $\iota: G_{D}/G_{\gamma(D)} \to \left( \mathbf{G}(\OS)\right)_{D}/ \left( \mathbf{G}(\OS)\right)_{\gamma(D)}$.
We identify $G_{D}/G_{\gamma(D)}$ with its image via $\iota$, so that we realizes $G_{D}/G_{\gamma(D)}$ as a subgroup of $\left( \mathbf{G}(\OS)\right)_{D}/ \left( \mathbf{G}(\OS)\right)_{\gamma(D)}$.
Since $G$ has finite index in $\mathbf{G}(\OS)$, the group $G_{D}$ has finite index in $\left( \mathbf{G}(\OS)\right)_{D}$, whence $G_{D}/G_{\gamma(D)}$ has a finite index in $\left( \mathbf{G}(\OS)\right)_{D}/ \left( \mathbf{G}(\OS)\right)_{\gamma(D)}$.
This implies that $r(D,G)=r_{0,D}$, since the respective free $\mathbb{Z}$-modules are commensurable.
Thus, Statement~(3) follows.
\end{proof}

\begin{corollary}\label{coro O_s principal and G normal}
Assume that $\mathcal{O}_{\mathcal{S}}$ is a principal ideal domain. 
Let $G$ be a normal subgroup of $\mathbf{G}(\mathcal{O}_{\mathcal{S}})$.
Then, for any $\D \in \SecD$:
\begin{itemize}
\item[(1)] $G_D$ is $\mathbf{G}(\mathcal{O}_{\mathcal{S}})$-conjugate to $\mathbf{B}(k) \cap G$,
\item[(2)] $U(G_D)$ is $\mathbf{G}(\mathcal{O}_{\mathcal{S}})$-conjugate to $\mathbf{U}^{+}(k) \cap G$,
\item[(3)] $T(G_D)$ is isomorphic to the direct product of a subgroup of $\mathbf{T}(\mathbb{F})$ and $\mathbb{Z}^r$, where $r$ is at most $\mathbf{t} \cdot (\sharp \mathcal{S} -1)$, with equality when $\mathrm{Pic}^{0}(\mathcal{C})$ is a torsion group. In particular, we have $r =  \mathbf{t} \cdot (\sharp \mathcal{S} -1)$ when $\mathbb{F}$ is finite.
\end{itemize}
Moreover, the set $\mathfrak{U}/G$ of $G$-conjugacy classes of maximal unipotent subgroups in $G$ is in bijection with the double quotient $G \backslash \mathbf{G}(\mathcal{O}_{\mathcal{S}}) / \mathbf{B}(\mathcal{O}_{\mathcal{S}})$.
\end{corollary}

\begin{proof}
Since $\mathcal{O}_{\mathcal{S}}$ is a principal ideal domain, it follows from Theorem~\ref{main teo 3} that $\mathbf{G}(\mathcal{O}_{\mathcal{S}})$ acts transitively on $\SecD$. 
In other words, each $\D \in \SecD$ belongs to the $\mathbf{G}(\OS)$-orbit of $\D_0$.
Thus, Statement~(1) and~(2) follows from Lemma~\ref{lemma G normal and its subgroups}~(1)-(2). 
Moreover, it follows from Lemma~\ref{lemma G normal and its subgroups}~(3) that $T(G_D)$ is isomorphic to the direct product of a subgroup of $\mathbf{T}(\mathbb{F})$ and a $\mathbb{Z}$-free group of rank $r_{0,D_0}=r_{D_0}$. Thus, Statement~(3) follows from Example~\ref{ex for D_0}.

Now, we prove that $\mathfrak{U}/G$ is in bijection with $G \backslash \mathbf{G}(\mathcal{O}_{\mathcal{S}}) / \mathbf{B}(\mathcal{O}_{\mathcal{S}})$.
Since $\mathbf{G}(\mathcal{O}_{\mathcal{S}})$ acts transitively on $\SecD$, and $\mathrm{Stab}_{\mathbf{G}(\mathcal{O}_{\mathcal{S}})}(\D_0)=\mathbf{G}(\mathcal{O}_{\mathcal{S}}) \cap \mathbf{B}(k)=\mathbf{B}(\mathcal{O}_{\mathcal{S}})$, there exists a $\mathbf{G}(\OS)$-equivariant bijection between $\SecD$ and $\mathbf{G}(\mathcal{O}_{\mathcal{S}}) / \mathbf{B}(\mathcal{O}_{\mathcal{S}})$.
Therefore, the set of $G$-orbits on $\SecD$ is in bijection with $G \backslash \mathbf{G}(\mathcal{O}_{\mathcal{S}}) / \mathbf{B}(\mathcal{O}_{\mathcal{S}})$, whence the result follows.
\end{proof}

\section{Applications to principal congruence subgroups}\label{section examples}

Assuming that $\mathbb{F}$ is a finite field of characteristic $p$, in this section,
we present some explicit examples on the description of maximal unipotent subgroups introduced in Theorems~\ref{main teo 1} and~\ref{main teo 2}.

Let $I$ be a proper ideal of $\mathcal{O}_{\mathcal{S}}$.
The principal congruence subgroup $\Gamma_I$ defined from $I$ is the kernel of the group homomorphism $\pi_I: \mathbf{G}(\mathcal{O}_{\mathcal{S}}) \to \mathbf{G}(\mathcal{O}_{\mathcal{S}}/I)$ induced by the projection $\pi_I: \mathcal{O}_{\mathcal{S}} \to \mathcal{O}_{\mathcal{S}}/I$.
A principal congruence subgroup of $\mathbf{G}(\OS)$ is a group of the form $\Gamma_I$, for some ideal $I\subseteq \mathcal{O}_{\mathcal{S}}$.
Such groups have finite index in $\mathbf{G}(\OS)$, since $\OS/I$ is finite whenever $\mathbb{F}$ is finite.
In particular, principal congruence subgroups are arithmetic subgroups.
Next result is an elementary extension of~\cite[Lemma 3.3]{MPSZ}, which is valid for rank one groups, to the context of higher rank.

\begin{lemma}\label{lemma torsion elements in Gamma}
Assume that $\mathbb{F}$ is finite of characteristic $p$.
The torsion of $\Gamma_I$ is $p$-primary.
\end{lemma}

\begin{proof}
According to \cite[\S 9.1.19(c)]{BT1}, there exists an injective homomorphism of $\mathbb{Z}$-groups  $\rho: \mathbf{G} \to \mathrm{SL}_{n,\mathbb{Z}}$. This is a faithful linear representation of $\mathbf{G}$. In particular, for each (commutative) ring $R$, the homomorphism $\rho$ induces an injective group homomorphism $\mathbf{G}(R) \to \mathrm{SL}_{n, \mathbb{Z}}(R)$, which, by abuse of notation, we denote by $\rho$.
Let $g \in \Gamma_{I}$ be a finite order element and let $P_{g}(T)$ be the characteristic polynomial of $\rho(g)$ over $k$.
When we apply the ring homomorphism $\pi_I: \OS \to \OS/I$ to each coefficient of $P_g(T)$, we obtain the a polynomial $\pi_I(P_g(T)) \in (\OS/I)[T]$.
Since $\pi_I: \OS \to \OS/I$ is a ring homomorphism, we have
$$\pi_I(P_g(T))= \pi_I\left(\mathrm{det}\left(\rho(g)-T \cdot \mathrm{id}\right)\right) = \mathrm{det}\left(\pi_I(\rho(g))-T \cdot \mathrm{id}\right).$$
Moreover, since $\pi_I(\rho(g))=\rho (\pi_I(g))$ and $\pi_I(g)=\mathrm{id}$, we get
$$\mathrm{det}\left(\pi_I(\rho(g))-T \cdot \mathrm{id}\right)=\mathrm{det}\left(\rho(\pi_I(g))-T \cdot \mathrm{id}\right)=\det\left(\mathrm{id}-T \cdot \mathrm{id}\right)=(1-T)^n.$$
Then $\pi_I(P_g(T))=(1-T)^n$.
But, since $g$ is a torsion element, $g^m=\mathrm{id}$, for some $m \in \mathbb{Z}$.
Therefore $\rho(g)^m=\mathrm{id}$, whence we deduce that each eigenvalue of $\rho(g)$ is a root of unity. 
In particular, each coefficient of $P_g(T)$ lies in the algebraic closure of $\mathbb{F}$ in $k$,
which is $\mathbb{F}$ itself since $\mathcal{C}$ is assumed geometrically integral over $\mathbb{F}$.
Thus, the polynomial $P_g(T)$ belongs to $\mathbb{F}[T]$. 
In particular, each coefficient of $P_g(T)$ lies in $\mathbb{F}$, and we have $P_g(T)=\pi_I(P_g(T))=(T-1)^n$, since $\mathbb{F}$ is a subring of $\OS/I$.
Hence, the matrix $\rho(g)$ is unipotent.
Since $k$ has characteristic $p$, any unipotent element of $\mathrm{SL}_{n,\mathbb{Z}}(k)$ has $p$-power order.
Therefore, $\rho(g)^{p^t}=\mathrm{id}$, for some $t \in \mathbb{Z}$.
Thus, we conclude that $g^{p^t}=\mathrm{id}$.
\end{proof}

\begin{corollary}\label{unipotent of principal cong subgroups when S={p}}
Assume that $\mathbb{F}$ is finite of characteristic $p$. 
Assume also that $\mathcal{S}=\lbrace \p \rbrace$.
Fix a set $\lbrace \D_{\sigma}: \sigma \in \Sigma \rbrace$ of representatives of the $G$-orbits in $\SecD$.
Then:
\begin{itemize}
\item[(1)] $\mathfrak{U}=\left \lbrace \mathrm{Stab}_{\Gamma_I}(\D) : \D \in \SecD \right\rbrace$ is the set of maximal unipotent subgroup of $\Gamma_I$, and
\item[(2)] $\mathfrak{U}/\Gamma_I=\left \lbrace \mathrm{Stab}_{\Gamma_I}(\D_{\sigma}) : \sigma \in \Sigma \right\rbrace$ is a set of representatives of the $\Gamma_I$-conjugacy classes in $\mathfrak{U}$.
\end{itemize}
\end{corollary}

\begin{proof}
Since $\mathbb{F}$ is finite of characteristic $p$, the torsion of $\Gamma_I$ is $p$-primary according to Lemma~\ref{lemma torsion elements in Gamma}. Since $\mathcal{S}=\lbrace \p \rbrace$, it follows from Theorem~\ref{main teo 2}(5) that, for any $\D \in \SecD$, we have $U((\Gamma_I)_D)=(\Gamma_I)_D=\mathrm{Stab}_{\Gamma_I}(\D)$. Then, the result follows from Theorem~\ref{main teo 1}(2)-(3).
\end{proof}

\begin{example}
Statements~(1) and~(2) of Corollary~\ref{unipotent of principal cong subgroups when S={p}} do not hold for arbitrary finite subsets $\mathcal{S}$.
For instance, let us take $\mathcal{C}=\mathbb{P}^1_{\mathbb{F}}$ and $\mathcal{S}=\lbrace 0, \infty \rbrace$, so that $\OS=\mathbb{F}[t,t^{-1}]$. 
Let $I=(t-1)\OS$, so that $\pi_I:\OS \to \OS/I$ corresponds to the ring homomorphism defined by the evaluation at $t=1$.
Let $\mathbf{G}=\mathrm{SL}_n$ and let $\Gamma_{I}$ be the corresponding congruence subgroup. 
We identify $\mathbf{T}$ (resp. $\mathbf{B}$) with the diagonal (resp. upper triangular) subgroup of $\mathrm{SL}_n$.
Let
$$T:=\lbrace \mathrm{diag}(t^{m_1}, \cdots, t^{m_n}) : m_1+ \cdots+ m_n=0, m_1, \cdots, m_n \in \mathbb{Z} \rbrace.$$
Since $\pi_I(t)=1$, we have $T \subseteq \Gamma_I$.
Moreover, since $T \subset \mathbf{B}(k)$, we have $T \subseteq \mathrm{Stab}_{\Gamma_I}(\D_0)$. This proves that $\mathrm{Stab}_{\Gamma_I}(\D_0)$ is not unipotent. Thus, Statements~(1) and~(2) of Corollary~\ref{unipotent of principal cong subgroups when S={p}} do not hold in this context.
\end{example}

\begin{corollary}\label{corollary principal cong sub and principal ideal domains}
Assume that $\mathbb{F}$ is finite of characteristic $p$.
Assume also that $\mathcal{O}_{\mathcal{S}}$ is a principal ideal domain.
Let $\mathfrak{U}$ and $\mathfrak{U}/\Gamma_I$ as in Theorem~\ref{main teo 1}. Then
\begin{itemize}
\item[(1)] Any $U \in \mathfrak{U}$ is $\mathbf{G}(\mathcal{O}_{\mathcal{S}})$-conjugate to $\Gamma_I^{+}:=\mathbf{U}^{+}(k) \cap \Gamma_I$, and
\item[(2)] $\mathfrak{U}/\Gamma_I$ is in bijection with $\Gamma_I \backslash \mathbf{G}(\mathcal{O}_{\mathcal{S}}) / \mathbf{B}(\mathcal{O}_{\mathcal{S}})$.
\end{itemize}
\end{corollary}

\begin{proof}
Since $\Gamma_I$ is a normal subgroup of $\mathbf{G}(\OS)$, Statement~(1) follows from Corollary~\ref{coro O_s principal and G normal}~(2).
Moreover, it also follows from Corollary~\ref{coro O_s principal and G normal} that $\mathfrak{U}/\Gamma_I$ is in bijection with $\Gamma_I \backslash \mathbf{G}(\mathcal{O}_{\mathcal{S}}) / \mathbf{B}(\mathcal{O}_{\mathcal{S}})$, which concludes the proof.
\end{proof}

\begin{example}\label{ex max unip sln}
Let $\mathbf{G}=\mathrm{SL}_n$. The subgroup $\mathrm{B}_n$ (resp. $\mathrm{D}_n$, resp. $\mathrm{U}_n$) of upper triangular (resp. diagonal, resp. unipotent upper triangular) matrices in $\mathrm{SL}_n$ is a Borel (resp. a maximal torus, resp. a maximal unipotent) subgroup of $\mathrm{SL}_n$ defined over $\mathbb{Z}$.
Let $\varTheta_I$ be the set of nilpotent upper triangular matrices with coefficients in $I$.
The group $\Gamma_I^{+}$ equals $\mathrm{id}+\varTheta_I=\lbrace \mathrm{id}+ \theta: \theta \in \varTheta_I \rbrace$.
By a straightforward computation in $\mathrm{SL}_n$, the group $\Gamma_I^{+}$ is the group generated by $\left \lbrace \theta_{\alpha}(I): \alpha \in \Phi^{+} \right\rbrace$.
This provides an example of the group $\Gamma_I^{+}$ introduced in Corollary~\ref{corollary principal cong sub and principal ideal domains}~(1).

Since the group homomorphism $\pi_I: \mathrm{SL}_n(\mathcal{O}_{\mathcal{S}}) \to \mathrm{SL}_n(\mathcal{O}_{\mathcal{S}}/I)$ is surjective, we have  $\Gamma_I \backslash \mathrm{SL}_n(\mathcal{O}_{\mathcal{S}}) \cong \mathrm{SL}_n(\mathcal{O}_{\mathcal{S}}/I)$.
Moreover, since $\left(\Gamma_I \cap \mathrm{B}_n(\mathcal{O}_{\mathcal{S}}) \right) \backslash \mathrm{B}_n(\mathcal{O}_{\mathcal{S}}) \cong \mathrm{B}_n(\mathcal{O}_{\mathcal{S}}/I)$, we get that $\Gamma_I \backslash \mathrm{SL}_n(\mathcal{O}_{\mathcal{S}}) / \mathrm{B}_n(\mathcal{O}_{\mathcal{S}})$ is in bijection with $\mathrm{SL}_n(\mathcal{O}_{\mathcal{S}}/I) / \mathrm{B}_n(\mathcal{O}_{\mathcal{S}}/I)$.
Therefore, it follows from Corollary~\ref{corollary principal cong sub and principal ideal domains}~(2) that the set $\mathfrak{U}/\Gamma_I$, consisting in the $\Gamma_I$-conjugacy classes of maximal unipotent subgroup of $\Gamma_I$, is in bijection with $\mathrm{SL}_n(\mathcal{O}_{\mathcal{S}}/I) / \mathrm{B}_n(\mathcal{O}_{\mathcal{S}}/I)$.
This bijection allows us to do some explicit computations on $\mathfrak{U}/\Gamma_I$.
In particular, when $I$ is assumed to be a prime ideal, the ring $\mathbb{F}':=\mathcal{O}_{\mathcal{S}}/I$ is a field, since $\OS$ is a Dedekind domain. 
In that case, the exact sequence of algebraic varieties
\[
1 \rightarrow \mathrm{B}_n \xrightarrow{\iota} \mathrm{SL}_n \xrightarrow{p} \mathrm{SL}_n/\mathrm{B}_n \rightarrow 1,
\]
induces the following long exact sequence (cf.~\cite[Ch. III, \S 4, 4.6]{DG})
\[
1 \to \mathrm{B}_n(\mathbb{F}') \rightarrow \mathrm{SL}_n(\mathbb{F}') \rightarrow (\mathrm{SL}_n/\mathrm{B}_n)(\mathbb{F}') \rightarrow H^1_{\text{\'et}}\big(\operatorname{Spec}(\mathbb{F}'), \mathrm{B}_n\big) \xrightarrow{} H^1_{\text{\'et}}\big(\operatorname{Spec}(\mathbb{F}'),\mathrm{SL}_n\big).
\]
But $H^1_{\text{\'et}}\big(\operatorname{Spec}(\mathbb{F}'), \mathrm{B}_n\big)= H^1_{\text{\'et}}\big(\operatorname{Spec}(\mathbb{F}'), \mathrm{D}_n\big) \cong H^1_{\text{\'et}}\big(\operatorname{Spec}(\mathbb{F}'), \mathbb{G}_m\big)^{n-1}=\lbrace 0 \rbrace$, according to~\cite[Exp.~XXVI, Cor. 2.3]{SGA3-3} and Hilbert's Theorem 90.
Then, we get that $\mathrm{SL}_n(\mathbb{F}')/\mathrm{B}_n(\mathbb{F}')$ is in bijection with $(\mathrm{SL}_n/\mathrm{B}_n)(\mathbb{F}')$.
Thus, the set $\mathfrak{U}/\Gamma_I$ is in bijection with $(\mathrm{SL}_n/\mathrm{B}_n)(\mathbb{F}')$, in the case of a prime ideal $I$.
In other words, the set of $\Gamma_I$-conjugacy classes of maximal unipotent subgroups of $\Gamma_I$ is in bijection with the set of $\mathbb{F}'$-points of the Borel variety $\mathrm{SL}_n/\mathrm{B}_n$.
\end{example}

\section*{Acknowledgements}
The first author was partially supported by Anid-Conicyt by the Postdoctoral fellowship No $74220027$.


\bibliographystyle{amsalpha}
\bibliography{refs.bib}

\end{document}